\RequirePackage{fix-cm}
\documentclass[smallextended]{svjour3}       % onecolumn (second format)
\smartqed  % flush right qed marks, e.g. at end of proof
\usepackage{graphicx}

\usepackage{amsmath}
\usepackage{amssymb}
\usepackage{url}
\usepackage{tikz}
\usepackage{enumerate}
\usepackage{paralist}
\usepackage{multicol}
\usepackage{mathptmx}      % use Times fonts if available on your TeX system

\numberwithin{theorem}{section}
\newtheorem{cclaim}[theorem]{Claim}
\newtheorem{ccorollary}[theorem]{Corollary}
\newtheorem{ddefinition}[theorem]{Definition}
\newtheorem{llemma}[theorem]{Lemma}
\newtheorem{oobservation}[theorem]{Observation}
\journalname{Graphs and Combinatorics}

\usepackage[shortcuts]{extdash}  
 
\begin{document}

\title{Which graphs occur as $\gamma$-graphs?
\thanks{M.~Devos, A.~Dyck and J.~Jedwab are supported by NSERC.\\The results of this paper form part of the Master's thesis of A. Dyck \cite{dyck-masters}, who presented them in part at the CanaDAM 2017 conference in Toronto, ON.}
}

\author{Matt DeVos \and Adam Dyck \and Jonathan Jedwab \and Samuel Simon
}

\authorrunning{DeVos Dyck Jedwab Simon} 

\institute{M. DeVos, A. Dyck, J. Jedwab, S. Simon \at
Department of Mathematics,
Simon Fraser University, 8888 University Drive, Burnaby BC V5A 1S6, Canada \\
              \email{mdevos@sfu.ca, ard9@sfu.ca, jed@sfu.ca, ssimon@sfu.ca} %
}

\date{Received: date / Accepted: date}

\maketitle

\begin{abstract}
The $\gamma$-graph of a graph $G$ is the graph whose vertices are labelled by the minimum dominating sets of $G$, in which two vertices are adjacent when their corresponding minimum dominating sets (each of size $\gamma(G)$) intersect in a set of size $\gamma(G)-1$.
We extend the notion of a $\gamma$-graph from distance\-/1\-/domination to distance\-/$d$\-/domination, and ask which graphs $H$ occur as $\gamma$-graphs for a given value of~$d \ge 1$.
We show that, for all $d$, the answer depends only on whether the vertices of $H$ admit a labelling consistent with the adjacency condition for a conventional $\gamma$-graph.
This result relies on an explicit construction for a graph having an arbitrary prescribed set of minimum distance\-/$d$\-/dominating sets.
We then completely determine the graphs that admit such a labelling among the wheel graphs, the fan graphs, and the graphs on at most six vertices.
We connect the question of whether a graph admits such a labelling with previous work on induced subgraphs of Johnson graphs.

\end{abstract}

\usetikzlibrary{shadows}

%Node styles
\tikzstyle{n} = [shape=circle, minimum size=0.23cm, text=black, draw=black!150, bottom color=black!150, text width=0.0001cm, inner sep=0pt]
\tikzstyle{nwhite} = [shape=circle, minimum size=0.3cm, text=black, draw=black!150, top color=white, bottom color=white, text width=0.0001cm, inner sep=0pt]
% %

%Edge style
\tikzstyle{e} = [thick]

% % % % % % % % % % % % % % % % % % % % % 

\showboxdepth=\maxdimen
\showboxbreadth=\maxdimen

\section{Introduction} 
\label{sec:intro}

In this paper we consider only finite, loop-free, undirected graphs $G$ without multiple edges.
Our main object of study is the $\gamma_d$-graph of a graph $G$, which we introduce via the following three definitions.

\begin{ddefinition} \label{def:distddom}
Let $G$ be a graph, and let $S$ and $T$ be subsets of the vertex set $V(G)$ of~$G$.
The set $S$ \emph{distance\-/$d$\-/dominates} $T$ if every vertex of $T$ is within distance $d$ in $G$ of some vertex in~$S$. In the case $T = V(G)$, the subset $S$ is a \emph{distance\-/$d$\-/dominating set} of $G$.
\end{ddefinition}

\begin{ddefinition} \label{def:distdgammaset}
A \emph{minimum distance\-/$d$\-/dominating set} of a graph $G$ is a distance\-/$d$\-/dominating set of smallest size, and this size is the \emph{distance\-/$d$\-/domination number} $\gamma_d(G)$ of~$G$.
\end{ddefinition}

These definitions reduce to well-studied domination notions when $d=1$:
a distance\-/1\-/dominating set is a dominating set;
a minimum distance\-/1\-/dominating set is a minimum dominating set;
and the distance\-/1\-/domination number $\gamma_1(G)$ is the domination number~$\gamma(G)$. 
The study of domination in graphs spans more than fifty years, with early interpretations that include the number of queens required to access every square of a chessboard \cite{ore}, the strength of surveillance in a network \cite{berge}, and network communications \cite{liu}. The modern study of domination has connections to game theory, coding theory, and matching theory; see \cite{fundamentals} and~\cite{fundamentalsadvanced} for extensive background.
The extension of domination notions to the cases $d>1$ in Definitions~\ref{def:distddom} and~\ref{def:distdgammaset} follows \cite{distancedom} and~\cite{henning}, for example. 

\begin{ddefinition} \label{def:gammaddotg}
The \emph{$\gamma_d$-graph} $\gamma_d \cdot G$ of a graph $G$ has vertices labelled by the minimum distance\-/$d$\-/dominating sets of $G$, and an edge joining two vertices if and only if their corresponding labels intersect in a set of size $\gamma_d(G) - 1$.
\end{ddefinition}

The case $d=1$ of Definition~\ref{def:gammaddotg} corresponds to the $\gamma$-graph $\gamma \cdot G$, introduced by Subramanian and Sridharan~\cite{ssgamma} and subsequently studied in~\cite{sstrees},~\cite{lakshmanan},~\cite{sarinduce},~\cite{bien}. 
We believe that the generalisation of the $\gamma$-graph in Definition~\ref{def:gammaddotg} to cases $d>1$ is new.
(An alternative definition of a $\gamma$-graph, written $G(\gamma)$ and studied in \cite{connelly}, \cite{fricke}, \cite{edwards}, \cite{teshima}, imposes an additional restriction on the edges of the $\gamma$-graph; we do not consider that definition in this paper.) 
See Figure~\ref{fig:gamma1gamma2} for an example of a graph $G$ and its $\gamma_1$-graph and~$\gamma_2$-graph.

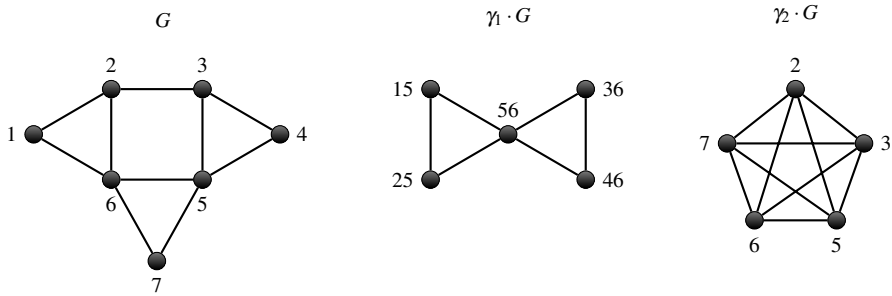
\begin{figure}
\begin{center}
\begin{tikzpicture}[scale=0.6]
\node [n, label=left: {$1$}, label={[xshift=1.7cm, yshift=1.2cm]$G$}] (v1) at (0.3,0) {};
\node [n, label=above: {$2$}] (v2) at (2,1) {};
\node [n, label=above: {$3$}] (v3) at (4,1) {};
\node [n, label=right: {$4$}] (v4) at (5.7,0) {};
\node [n, label=below: {$5$}] (v5) at (4,-1) {};
\node [n, label=below: {$6$}] (v6) at (2,-1) {};
\node [n, label=below: {$7$}] (v7) at (3,-2.8) {};
\draw [e] (v1) -- (v2) -- (v3) -- (v4) -- (v5) -- (v6) -- (v1);
\draw [e] (v6) -- (v7) -- (v5);
\draw [e] (v2) -- (v6);
\draw [e] (v3) -- (v5);

\node [n, label=left: {$15$}] (v15) at (9,1) {};
\node [n, label=left: {$25$}] (v25) at (9,-1) {};
\node [n, label=above: {$56$}, label={[xshift=0cm, yshift=1.2cm]$\gamma_1 \cdot G$}] (v56) at (10.7,0) {};
\node [n, label=right: {$36$}] (v36) at (12.4,1) {};
\node [n, label=right: {$46$}] (v46) at (12.4,-1) {};
\draw [e] (v56) -- (v25) -- (v15) -- (v56) -- (v36) -- (v46) -- (v56);

\node [n, label=above: {$2$}, label={[xshift=0cm, yshift=0.65cm]$\gamma_2 \cdot G$}] (v22) at (17,1) {};
\node [n, label=right: {$3$}] (v32) at (18.5,-0.2) {};
\node [n, label=below: {$5$}] (v52) at (17.9,-1.9) {};
\node [n, label=below: {$6$}] (v62) at (16.1,-1.9) {};
\node [n, label=left: {$7$}] (v72) at (15.5,-0.2) {};
\draw [e] (v22) -- (v32) -- (v52) -- (v62) -- (v72) -- (v22) -- (v52) -- (v72) -- (v32) -- (v62) -- (v22);
\end{tikzpicture}
\caption{A graph $G$, its $\gamma_1$-graph, and its $\gamma_2$-graph.} 
\label{fig:gamma1gamma2}
\end{center}
\end{figure}

We say that a graph $H$ is \emph{$d$-realisable} if there exists a graph $G$ for which $H = \gamma_d \cdot G$; otherwise $H$ is \emph{$d$-unrealisable}. 
A graph is \emph{minimally $d$-unrealisable} if it is $d$-unrealisable but every proper induced subgraph is $d$-realisable.
See Figure~\ref{fig:2-unrealisable} for an example of a 2-unrealisable and minimally 2-unrealisable graph.
The central objective is:
\begin{align}
& \mbox{Determine, for given $d$, which graphs $H$ are $d$-realisable and which} \nonumber \\[-0.5ex]
& \mbox{are minimally $d$-unrealisable}.  \label{central-obj}
\end{align}

\begin{figure}
\begin{center}
\begin{tikzpicture}[scale=0.9]
\node [n, label=below: {$ $}] (v1) at (0,0) {};
\node [n, label=left: {$ $}] (v2) at (-1.5,1.5) {};
\node [n, label=right: {$ $}] (v3) at (0,1.5) {};
\node [n, label=below: {$ $}] (v4) at (1.5,1.5) {};
\node [n, label={[yshift=0.5cm]$G$}] (v5) at (0,3) {};
\node [n, label=above: {$ $}] (v6) at (3,1.5) {};
\draw [e] (v6) -- (v4) -- (v1) -- (v2) -- (v5) -- (v4);
\draw [e] (v2) -- (v3) -- (v4);

\node [n, label=below: {$ $}] (v11) at (7,0) {};
\node [n, label=left: {$ $}] (v21) at (5.5,1.5) {};
\node [n, label=right: {$ $}] (v31) at (7,1.5) {};
\node [n, label=right: {$ $}] (v41) at (8.5,1.5) {};
\node [n, label={[yshift=0.5cm]$K_{2,3}$}] (v51) at (7,3) {};
\draw [e] (v21) -- (v51) -- (v41) -- (v11) -- (v21) -- (v31) -- (v41);
\end{tikzpicture}
\caption{The graph $G$ is 2-unrealisable, and it contains the induced subgraph $K_{2,3}$ which is minimally 2-unrealisable (as established in Corollary~\ref{cor:realisegammaddotg} and Theorem~\ref{thm:minunlab5}).}
\label{fig:2-unrealisable}
\end{center}
\end{figure}
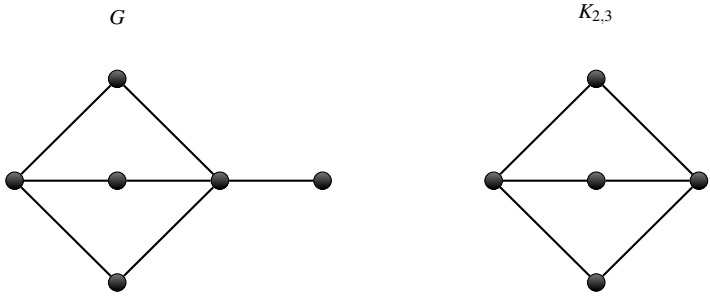

We say that a graph $H$ is \emph{labellable} if, for some positive integer $k$, the vertices of $H$ can be labelled by distinct $k$-subsets of $\{1,2,3,\dots\}$ such that two vertices are adjacent if and only if their corresponding labels intersect in a set of size $k - 1$; otherwise $H$ is \emph{unlabellable}.
Given a labellable graph $H$, neither its labelling nor the associated integer $k$ are unique: adding a new symbol to each of the vertex labels increases $k$ by one. The following observation is immediate.
\begin{oobservation}
\label{obs:induced}
Each induced subgraph of a labellable graph is labellable.
\end{oobservation}

A graph that is $d$-realisable for some positive integer $d$ is necessarily labellable. Our main result (Corollary~\ref{cor:realisegammaddotg} below) is that the converse holds for every~$d$, which we prove using the following theorem. We consider it very surprising that there is such a simple characterisation of when a graph is $d$-realisable. 

\begin{theorem} \label{thm:realisegammaddotg}
Let $k$ and $d$ be positive integers, and let $\mathcal{D}$ be a nonempty set of $k$-subsets of $\{1,2,3,\dots\}$. Then there is a graph $G$ whose minimum distance\-/$d$\-/dominating sets are the elements of~$\mathcal{D}$.
\end{theorem}

\begin{ccorollary} 
\label{cor:realisegammaddotg}
\mbox{}
\begin{enumerate}[$(i)$]
\item
A graph $H$ is $d$-realisable for every positive integer $d$ if and only if it is labellable.

\item
A graph $H$ is $d$-unrealisable for every positive integer $d$ if and only if it is unlabellable. 

\item
A graph $H$ is minimally $d$-unrealisable for every positive integer $d$ if and only if it is unlabellable but every proper induced subgraph is labellable.
\end{enumerate}

\end{ccorollary}

\begin{proof}
Let $d$ be a positive integer. We shall show that a graph $H$ is $d$-realisable if and only if it is labellable, which implies each of $(i)$, $(ii)$, and $(iii)$.

If $H$ is $d$-realisable with respect to a graph $G$ then it is labellable using the minimum distance\-/$d$\-/dominating sets of $G$.
Conversely, suppose that $H$ is labellable. Then for some positive integer $k$ the vertices of $H$ can be labelled by distinct $k$-subsets of $\{1,2,3\dots\}$ such that two vertices are adjacent if and only if their corresponding labels intersect in a set of size $k-1$. Then by Theorem~\ref{thm:realisegammaddotg} there is a graph $G$ whose minimum distance\-/$d$\-/dominating sets are these $k$-subsets. Therefore $H = \gamma_d \cdot G$ and so $H$ is $d$-realisable.
\qed
\end{proof}

\noindent In view of Observation~\ref{obs:induced} and Corollary~\ref{cor:realisegammaddotg}, we shall say that a graph $H$ is \emph{minimally unlabellable} if it is unlabellable but every proper induced subgraph is labellable. This allows us to rephrase the central objective \eqref{central-obj} as:
\begin{equation}
\label{central-obj2}
\mbox{Determine which graphs are labellable and which are minimally unlabellable}.
\end{equation}
Using the crucial insight that \eqref{central-obj} is equivalent to \eqref{central-obj2}, we shall simplify, unify, and extend many results that were previously stated and proved (often only with considerable effort) in terms of $\gamma_1$-graphs.

We now describe some relationships with previous work that uses different terminology. 

\begin{ddefinition}
\label{def:Johnson}
For positive integers $k$ and $n$ satisfying $k \le n$, the \emph{Johnson graph} $J(n,k)$ has vertices labelled by the $k$-subsets of $\{1,2,\dots,n\}$, and an edge joining two vertices if and only if their corresponding labels intersect in a set of size~$k-1$.
\end{ddefinition}
\noindent
Johnson graphs are well-studied as distance-regular graphs~\cite{brouwercohenneumaier}, in quantum probability~\cite{horaobata}, and in spectral analysis~\cite{krebsshaheen}; see~\cite{handbook} for further background.
It follows from Definition~\ref{def:Johnson} that
a graph $G$ is labellable if and only if it is (isomorphic to) an induced subgraph of a Johnson graph.
The results of \cite{naimishaw},~\cite{malikali}, and~\cite{malikalidelta}, concerning which graphs occur as an induced subgraph of a Johnson graph, can therefore be equivalently phrased as results on which graphs are labellable.
In this paper we extend many of these previous results.

The special case $d=1$ of Theorem~\ref{thm:realisegammaddotg} was established by Honkala, Hudry and Lobstein~\cite[Theorem~2]{hhldom} using the language of optimal dominating codes in graphs; we shall show that in this special case our construction proving Theorem~\ref{thm:realisegammaddotg} is simpler and, in some cases, much more economical.
Although these authors developed certain generalisations of the case $d=1$ of Theorem~\ref{thm:realisegammaddotg} in a later paper~\cite{hhltwinfree}, to our knowledge the cases $d>1$ of Theorem~\ref{thm:realisegammaddotg} (and therefore the cases $d>1$ of Corollary~\ref{cor:realisegammaddotg}) are new. 
Honkala, Hudry and Lobstein interpreted their result on optimal dominating codes \cite[Theorem~2]{hhldom} in terms of induced subgraphs of Johnson graphs, and cited results of \cite{naimishaw} on these graphs. They also defined a graph $\mathcal{N}(G)$ which is identical to $\gamma\cdot G$, but did not explicitly mention $\gamma$-graphs nor cite publications phrased in terms of $\gamma$-graphs.

We now outline the rest of the paper.
In Section~\ref{sec:proofmainthm} we give a constructive proof of Theorem~\ref{thm:realisegammaddotg}. 
In Section~\ref{sec:previousresults} we summarise previous proven and claimed results on which graphs are labellable and which are (minimally) unlabellable, and present counterexamples that disprove two of these claimed results.
In Section~\ref{sec:fll} we derive a series of lemmas that constrain the form of the labelling of an induced subgraph of a labellable graph~$G$, for use in subsequent sections.
In Section~\ref{sec:wheel-fan} we determine precisely which wheel graphs and which fan graphs are labellable, exhibiting an infinite family of minimally unlabellable graphs. 
In Section~\ref{sec:minunlab5} and Appendix~A we verify the previously known result that there are exactly four minimally unlabellable graphs on at most five vertices.
In Section~\ref{sec:minunlab6} and Appendix~B we prove that there are exactly four minimally unlabellable graphs on six vertices.
In Section~\ref{sec:claim7} we establish that a specific graph on seven vertices is minimally unlabellable.
We conclude in Section~\ref{sec:conc}.

\section{Proof of Theorem~\ref{thm:realisegammaddotg}}
\label{sec:proofmainthm}

In this section we give a constructive proof of Theorem~\ref{thm:realisegammaddotg}. 
The construction is illustrated in Figure~\ref{fig:gammaddotg}. 
We require the following concepts.

\begin{ddefinition}
A \emph{clutter} of a finite set $E$ is a collection $\mathcal{C}$ of subsets of $E$ for which no element of $\mathcal{C}$ contains another. The \emph{blocker} $b(\mathcal{C})$ of a clutter $\mathcal{C}$ is the collection of all minimal subsets of $E$ having nonempty intersection with each $C \in \mathcal{C}$. 
\end{ddefinition}

\begin{theorem}[{\cite[p.~301]{edmonds-fulkerson}}] 
\label{thm:blocker-blocker}
Let $\mathcal{C}$ be a clutter. Then $b(b(\mathcal{C})) = \mathcal{C}$.
\end{theorem}

\paragraph{Proof of Theorem~\ref{thm:realisegammaddotg}.}

Take $n = \Big|\bigcup\limits_{D \in \mathcal{D}} D\Big|$, and relabel if necessary so that each element of $\mathcal{D}$ is a subset of~$[n]:= \{1,2,\dots,n\}$. 
The set $\mathcal{D}$ is a clutter of~$[n]$.
Let $\mathcal{B}$ be the blocker of $\mathcal{D}$, namely the collection of all minimal subsets of $[n]$ having nonempty intersection with each $D \in \mathcal{D}$.
Construct the following graph~$G$.
\begin{enumerate}[Step 1.]
\item Initialise $G$ to be the complete graph $K_n$ and label its vertices $1, 2, \dots, n$.
\item For each $B \in \mathcal{B}$:
add new vertices $x_B$, $y_B$ to $G$;
add paths $P(x_B)$, $P(y_B)$ of length $d-1$ to $G$ that terminate in $x_B$, $y_B$, respectively;
and join $x_B$ and $y_B$ to each of the vertices of~$B$.
\end{enumerate}

We now prove the result by showing that the collection of minimum distance\-/$d$\-/dominating sets of $G$ equals~$\mathcal{D}$.

\begin{description}
\item[(a)]
\emph{A subset $D$ of $G$ is a distance\-/$d$\-/dominating set of $G$ if and only if it has nonempty intersection with each $B \in \mathcal{B}$.} 

A subset $D$ of $G$ that is a distance\-/$d$\-/dominating set of $G$ must have nonempty intersection with each $B \in \mathcal{B}$ in order that $D$ distance\-/$d$\-/dominates the pendant vertex of each path $P(x_B)$. 
Conversely, a subset $D$ of $G$ that has nonempty intersection with each $B \in \mathcal{B}$ distance-$1$-dominates the vertices labelled $1,2,\dots,n$ because $G$ was initialised to $K_n$ in Step~1, and distance-$d$-dominates the vertices of $P(x_B)$ and $P(y_B)$ for every $B \in \mathcal{B}$ by construction.

\item[(b)]
\emph{A minimum distance\-/$d$\-/dominating set of~$G$ is a subset of $[n]$.}

Suppose, for a contradiction, that for some $B \in \mathcal{B}$ a vertex $w$ in $P(x_B)$ is contained in a minimum distance\-/$d$\-/dominating set $D$ of~$G$.
By part (a), the set $D$ contains some vertex of $B$. This vertex of $B$ distance\-/$d$\-/dominates all vertices of $P(x_B)$ and $P(y_B)$ by construction, and is at least as close to each of the other vertices of $G$ as $w$ is. So we may obtain a smaller distance\-/$d$\-/dominating set than $D$ by removing $w$ from $D$, giving the required contradiction. 
\end{description}

By parts (a) and (b), the collection of minimum distance\-/$d$\-/dominating sets of $G$ is the collection of minimal subsets of $[n]$ having nonempty intersection with each $B \in \mathcal{B}$, namely the blocker of $\mathcal{B}$; by Theorem~\ref{thm:blocker-blocker}, this blocker equals~$\mathcal{D}$.
\qed
\vspace{2mm}
Figure~\ref{fig:gammaddotg} illustrates the proof given above, using the example of $d = 3$ and $\mathcal{D} = \big\{\{1,2,3\}, \{1,2,4\}\big\}$. We set $k=3$ and $n = 4$ and initialise $G$ to be $K_4$ with vertex labels $1, 2, 3, 4$.
The blocker of $\mathcal{D}$ is $\mathcal{B} = \big\{\{1\},\{2\},\{3,4\}\big\}$. For the element $\{1\}$ of $\mathcal{B}$, we add vertices $x_1$, $y_1$ to $G$, add paths $P(x_1)$, $P(y_1)$ of length 2 to $G$ that terminate in vertices $x_1$, $y_1$ respectively, and join $x_1$ and $y_1$ to the vertex~$1$. 
We repeat for each other element of~$\mathcal{B}$. The resulting graph $G$ has $22$ vertices and $26$ edges.
In general, the graph $G$ constructed according to the proof of Theorem~\ref{thm:realisegammaddotg} has
$n + 2d |\mathcal{B}|$ vertices and
$\binom{n}{2} + 2 \sum\limits_{B \in \mathcal{B}} |B| + 2(d-1)|\mathcal{B}|$ edges.

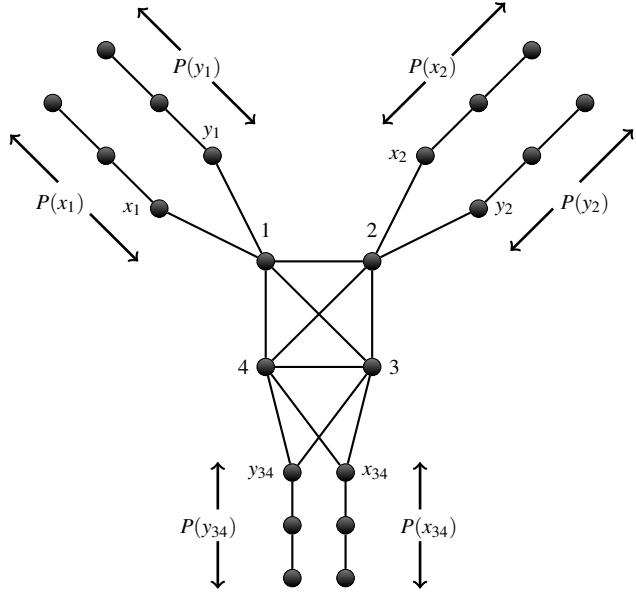
\begin{figure}
\centering
\begin{tikzpicture}[scale=0.7]
\node [n, label={[yshift=0.1cm]$1$}] (v1) at (0,2) {};
\node [n, label={[yshift=0.1cm]$2$}] (v2) at (2,2) {};
\node [n, label=right: {$3$}] (v3) at (2,0) {};
\node [n, label=left: {$4$}] (v4) at (0,0) {};
\draw [e] (v1) -- (v2) -- (v3) -- (v4) -- (v1) -- (v3);;
\draw [e] (v2) -- (v4);

\node [n, label=right: {$x_{34}$}] (x34a) at (1.5,-2) {};
\node [n, label={[xshift=1.1cm,yshift=-0.4cm]$P(x_{34})$}] (x34b) at (1.5,-3) {};
\node [n, label=right: {$ $}] (x34c) at (1.5,-4) {};
\draw[->,line width=1pt] (2.9,-2.7) to (2.9,-1.8);
\draw[->,line width=1pt] (2.9,-3.3) to (2.9,-4.2);

\node [n, label=left: {$y_{34}$}] (y34a) at (0.5,-2) {};
\node [n, label={[xshift=-1.1cm,yshift=-0.4cm]$P(y_{34})$}] (y34b) at (0.5,-3) {};
\node [n, label=left: {$ $}] (y34c) at (0.5,-4) {};
\draw[->,line width=1pt] (-0.9,-2.7) to (-0.9,-1.8);
\draw[->,line width=1pt] (-0.9,-3.3) to (-0.9,-4.2);
\draw [e] (x34c) -- (x34b) -- (x34a) -- (v3) -- (y34a) -- (y34b) -- (y34c);
\draw [e] (x34a) -- (v4) -- (y34a);

\node [n, label=left: {$x_{2}$}] (x2a) at (3,4) {};
\node [n, label=above: {$ $}] (x2b) at (4,5) {};
\node [n, label={[xshift=-1.3cm,yshift=-0.6cm]$P(x_2)$}] (x2c) at (5,6) {};
\draw[->,line width=1pt] (3.5,5.9) to (4.5,6.9);
\draw[->,line width=1pt] (3.0,5.4) to (2.2,4.6);

\node [n, label=right: {$y_{2}$}] (y2a) at (4,3) {};
\node [n, label=right: {$ $}] (y2b) at (5,4) {};
\node [n, label={[xshift=-0cm,yshift=-1.7cm]$P(y_2)$}] (y2c) at (6,5) {};
\draw[->,line width=1pt] (5.9,3.5) to (6.9,4.5);
\draw[->,line width=1pt] (5.4,3.0) to (4.6,2.2);
\draw [e] (x2c) -- (x2b) -- (x2a) -- (v2) -- (y2a) -- (y2b) -- (y2c);

\node [n, label=left: {$x_{1}$}] (x1a) at (-2,3) {};
\node [n, label=left: {$ $}] (x1b) at (-3,4) {};
\node [n, label={[xshift=0.1cm,yshift=-1.7cm]$P(x_1)$}] (x1c) at (-4,5) {};
\draw[->,line width=1pt] (-3.9,3.5) to (-4.8,4.4);
\draw[->,line width=1pt] (-3.3,2.9) to (-2.4,2.0);

\node [n, label=above: {$y_{1}$}] (y1a) at (-1,4) {};
\node [n, label=above: {$ $}] (y1b) at (-2,5) {};
\node [n, label={[xshift=1.2cm,yshift=-0.6cm]$P(y_1)$}] (y1c) at (-3,6) {};
\draw[->,line width=1pt] (-1.6,6.0) to (-2.4,6.8);
\draw[->,line width=1pt] (-1.0,5.4) to (-0.2,4.6);
\draw [e] (x1c) -- (x1b) -- (x1a) -- (v1) -- (y1a) -- (y1b) -- (y1c);
\end{tikzpicture}
\caption{The graph $G$ for $d = 3$ and $\mathcal{D} = \big\{ \{1,2,3\}, \{1,2,4\} \big\}$, as constructed in the proof of Theorem~\ref{thm:realisegammaddotg}. The blocker of $\mathcal{D}$ is $\mathcal{B} = \big\{ \{1\},\{2\},\{3,4\} \big\}$.}
\label{fig:gammaddotg}
\end{figure}

In the special case $d=1$, the constructed graph $G$ has $n+2|\mathcal{B}|$ vertices and 
$\binom{n}{2} + 2 \sum\limits_{B \in \mathcal{B}} |B|$ edges. 
This special case is also proved constructively in \cite[Theorem~2]{hhldom}, by means of a different graph containing
\[
n + (k + 1) \binom{n}{k - 1} + (k + 1)\left( \binom{n}{k} - |\mathcal{D}|\right)
\]
vertices and
\[
\binom{n}{2} + (k + 1)(n - k + 1) \binom{n}{k - 1} + (k + 1)(n - k)\left( \binom{n}{k} - |\mathcal{D}|\right)
\]
edges. The construction presented here is simpler and, in some cases, much more economical. For example, for $d=1$ and $\mathcal{D} = \big\{ \{1,2,3\}, \{1,2,4\} \big\}$ (giving $k=3$ and $n=4$ and $\mathcal{B} = \big\{ \{1\},\{2\},\{3,4\} \big\}$), the graph constructed here contains 10 vertices and 14 edges whereas the graph constructed according to the method of \cite[Theorem~2]{hhldom} contains 36 vertices and 62 edges
(see Figure~\ref{fig:mainthmexin1}).
For a further example, for $d=1$ and
$\mathcal{D} = \big\{ \{1,2,3,4\}, \{1,2,3,5\}, \{1,2,4,6\}, \{2,3,5,7\}, \{3,5,7,8\} \big\}$
(giving $k=4$ and $n=8$ and 
$\mathcal{B} = \big\{ \{1,3\},\! \{1,5\},\! \{1,7\},\! \{2,3\},\! \{2,5\},\! \{2,7\},\! \{2,8\},\! \{3,4\},\! \{3,6\},\! \{4,5\} \big\}$),
the graph constructed here contains 28 vertices and 68 edges whereas the graph constructed according to the method of \cite[Theorem~2]{hhldom} contains 613 vertices and 2728 edges.

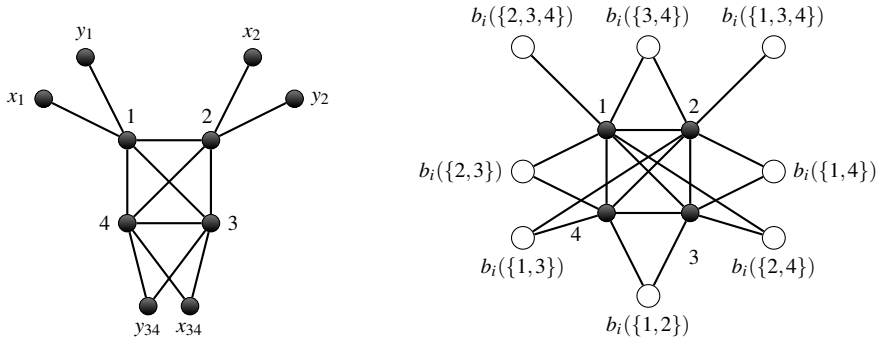
\begin{figure}
\begin{center}

\begin{tikzpicture}[scale=0.55]
\node [n, label={[xshift=0.05cm]$1$}] (v1) at (0,2) {};
\node [n, label={[xshift=-0.05cm]$2$}] (v2) at (2,2) {};
\node [n, label=right: {$3$}] (v3) at (2,0) {};
\node [n, label=left: {$4$}] (v4) at (0,0) {};
\draw [e] (v1) -- (v2) -- (v3) -- (v4) -- (v1) -- (v3);;
\draw [e] (v2) -- (v4);

\node [n, label=below: {$x_{34}$}] (x12) at (1.5,-2) {};
\node [n, label=below: {$y_{34}$}] (y12) at (0.5,-2) {};
\draw [e] (x12) -- (v3) -- (y12) -- (v4) -- (x12);

\node [n, label=above: {$x_{2}$}] (x134) at (3,4) {};
\node [n, label=right: {$y_{2}$}] (y134) at (4,3) {};
\draw [e] (x134) -- (v2) -- (y134);

\node [n, label=left: {$x_{1}$}] (x234) at (-2,3) {};
\node [n, label=above: {$y_{1}$}] (y234) at (-1,4) {};
\draw [e] (x234) -- (v1) -- (y234);
\end{tikzpicture}
\hspace{3em}
\begin{tikzpicture}[scale=0.55]
\node [n, label={[xshift=-0.05cm]$1$}] (v1) at (0,2) {};
\node [n, label={[xshift=0.05cm]$2$}] (v2) at (2,2) {};
\node [n, label={[xshift=0.05cm, yshift=-0.9cm]$3$}] (v3) at (2,0) {};
\node [n, label={[xshift=-0.4cm, yshift=-0.6cm]$4$}] (v4) at (0,0) {};
\draw [e] (v1) -- (v2) -- (v3) -- (v4) -- (v1) -- (v3);;
\draw [e] (v2) -- (v4);

\node [nwhite, label=above: {$b_i(\{3,4\})$}] (x34) at (1,4) {};
\draw [e] (v2) -- (x34) -- (v1);

\node [nwhite, label=below: {$b_i(\{1,2\})$}] (x12) at (1,-2) {};
\draw [e] (v4) -- (x12) -- (v3);

\node [nwhite, label=right: {$b_i(\{1,4\})$}] (x14) at (4,1) {};
\draw [e] (v3) -- (x14) -- (v2);

\node [nwhite, label=left: {$b_i(\{2,3\})$}] (x23) at (-2,1) {};
\draw [e] (v4) -- (x23) -- (v1);

\node [nwhite, label=below: {$b_i(\{1,3\})$}] (x13) at (-2,-0.6) {};
\draw [e] (v4) -- (x13) -- (v2);

\node [nwhite, label=below: {$b_i(\{2,4\})$}] (x24) at (4,-0.6) {};
\draw [e] (v3) -- (x24) -- (v1);

\node [nwhite, label=above: {$b_i(\{1,3,4\})$}] (x134) at (4,4) {};
\draw [e] (x134) -- (v2);

\node [nwhite, label=above: {$b_i(\{2,3,4\})$}] (x234) at (-2,4) {};
\draw [e] (x234) -- (v1);
\end{tikzpicture}

\caption{The graph $G$ for $d=1$ and $\mathcal{D} = \big\{\{1,2,3\}, \{1,2,4\}\big\}$, as constructed in the proof of Theorem~\ref{thm:realisegammaddotg} (left) and in the proof of \cite[Theorem~2]{hhldom} (right, where each white vertex represents a set of four vertices).}
\label{fig:mainthmexin1}
\end{center}
\end{figure}

\section{Previous results}
\label{sec:previousresults}

In this section we summarise previous proven and claimed results on which graphs are labellable and which are (minimally) unlabellable, taken primarily from the literature on induced subgraphs of Johnson graphs.
In several cases, the consequences for $\gamma_1$-graphs implied by Corollary~\ref{cor:realisegammaddotg} were previously derived in the $\gamma$-graph literature only with considerable effort. Indeed, even the result for 1-realisable graphs implied by Observation~\ref{obs:induced} was proved in~\cite[Theorem~2.1]{sarinduce} only by means of a complicated construction involving many vertices and edges.

We begin with some general constructions of labellable graphs.
\begin{theorem}
\label{thm:genconstr}
\mbox{}
\begin{enumerate}
\item[(i)] \cite[Proposition~6]{naimishaw}
A graph is labellable if and only if each of its components is labellable.

\item[(ii)] \cite[Proposition~7]{naimishaw} 
The Cartesian product of two labellable graphs is labellable.

\item[(iii)] \cite[Proposition~5]{naimishaw} 
A graph $G$ is labellable if and only if the graph obtained by repeatedly deleting isolated vertices and pendant vertices from $G$ is empty or labellable.
\end{enumerate}
\end{theorem}

\noindent The result for 1-realisable graphs implied by Theorem~\ref{thm:genconstr}~(ii) was proved in \cite[Theorem~3.4]{lakshmanan}.

We next describe several infinite families of labellable graphs.
For a positive integer $n$, the \emph{hypercube graph} $Q_n$ has vertices labelled by the $2^n$ binary $n$-tuples, and an edge joining two vertices if and only if their corresponding labels differ in exactly one position.
For an integer $n \ge 3$, the \emph{prism graph} $\Pi_n$ on $2n$ vertices is formed by joining corresponding vertices of two cycle graphs~$C_n$.

\begin{theorem} 
\label{thm:labfams}
\mbox{}
\begin{enumerate}
\item[(i)] \cite[Proposition~4]{naimishaw} 
The complete graph $K_n$ on $n \ge 1$ vertices is labellable.

\item[(ii)] \cite[Proposition~4]{naimishaw} 
The cycle graph $C_n$ on $n \ge 3$ vertices is labellable.

\item[(iii)] (Corollary of Theorem~\ref{thm:genconstr}~(iii))
Every tree is labellable and every graph containing exactly one cycle is labellable.

\item[(iv)] \cite[Theorem~4.2]{malikali} 
For each positive integer $n$, the hypercube graph $Q_n$ is labellable.

\item[(v)] \cite[Theorem~4.1]{malikali} 
For each integer $n \ge 3$, the prism graph $\Pi_n$ is labellable.
\end{enumerate}
\end{theorem}

\noindent The result for 1-realisable graphs implied by Theorem~\ref{thm:labfams}~(ii), (iii), and~(iv) was proved in \cite[Theorem~2.4]{sstrees}, \cite[Theorem~2.6]{sstrees}, and \cite[proof of Lemma~2.2]{bien}, respectively.
Theorem~\ref{thm:labfams}~(iv) can alternatively be proved by noting that $Q_1$ is trivially labellable, regarding $Q_n$ for $n \ge 2$ as the Cartesian product of $Q_{n-1}$ and $Q_1$, and then using Theorem~\ref{thm:genconstr}~(ii).

We now specify all minimally unlabellable graphs on at most five vertices.
\begin{theorem}[{\cite[Theorem 2.7]{sstrees}, \cite[Theorem 2.3]{lakshmanan}, \cite[Theorem~2.1]{sarinduce}; independently \cite[Theorems~3.1 and~3.2]{malikalidelta}}] 
\label{thm:minunlab5}
There are exactly four minimally unlabellable graphs on at most five vertices, namely:

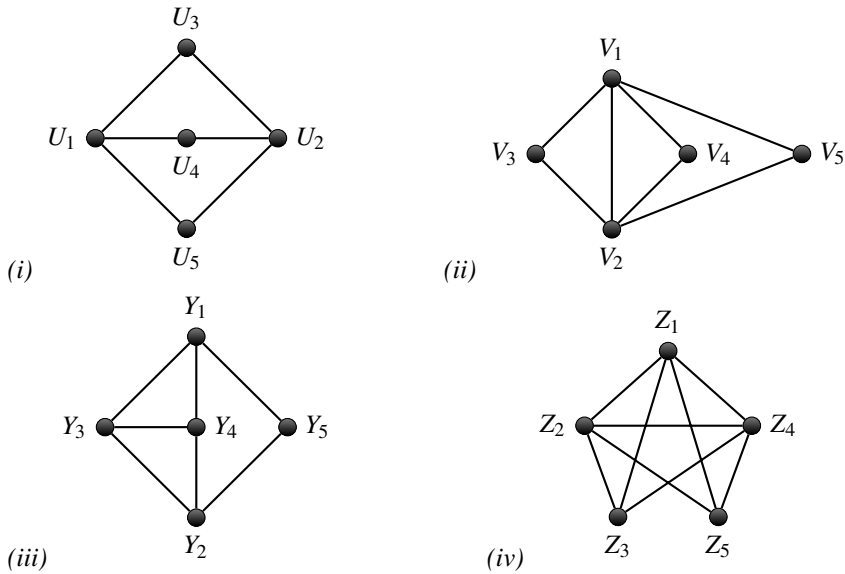
\begin{inparaenum}[(i)]
\item \begin{tikzpicture}[scale=0.8]
\node [n, label=below: {$U_5$}] (v1) at (0,0) {};
\node [n, label=left: {$U_1$}] (v2) at (-1.5,1.5) {};
\node [n, label=below: {$U_4$}] (v3) at (0,1.5) {};
\node [n, label=right: {$U_2$}] (v4) at (1.5,1.5) {};
\node [n, label=above: {$U_3$}] (v5) at (0,3) {};
\draw [e] (v1) -- (v2) -- (v5) -- (v4) -- (v1);
\draw [e] (v2) -- (v3) -- (v4);
\end{tikzpicture} \hspace{0.5in}
\item \begin{tikzpicture}[scale=1.0]
\node [n, label=below: {$V_2$}] (v1) at (0,0) {};
\node [n, label=left: {$V_3$}] (v2) at (-1,1) {};
\node [n, label=right: {$V_4$}] (v3) at (1,1) {};
\node [n, label=above: {$V_1$}] (v4) at (0,2) {};
\node [n, label=right: {$V_5$}] (v5) at (2.5,1) {};
\draw [e] (v1) -- (v2) -- (v4) -- (v3) -- (v1) -- (v4) -- (v5) -- (v1);
\end{tikzpicture}

\item \begin{tikzpicture}[scale=0.8]
\node [n, label=below: {$Y_2$}] (v1) at (0,0) {};
\node [n, label=left: {$Y_3$}] (v2) at (-1.5,1.5) {};
\node [n, label=right: {$Y_4$}] (v3) at (0,1.5) {};
\node [n, label=right: {$Y_5$}] (v4) at (1.5,1.5) {};
\node [n, label=above: {$Y_1$}] (v5) at (0,3) {};
\draw [e] (v1) -- (v3) -- (v5) -- (v4) -- (v1) -- (v2) -- (v5);
\draw [e] (v2) -- (v3);
\end{tikzpicture} \hspace{0.7in}
\item \begin{tikzpicture}[scale=1.1]
\node [n, label=left: {$Z_2$}] (v1) at (-1,0.1) {};
\node [n, label=right: {$Z_4$}] (v2) at (1,0.1) {};
\node [n, label=above: {$Z_1$}] (v3) at (0,1) {};
\node [n, label=below: {$Z_3$}] (v4) at (-0.6,-1) {};
\node [n, label=below: {$Z_5$}] (v5) at (0.6,-1) {};
\draw [e] (v5) -- (v2) -- (v4) -- (v3) -- (v5) -- (v1) -- (v2) -- (v3) -- (v1) --(v4);
\end{tikzpicture}
\end{inparaenum}
\end{theorem}
\vspace{1em}
Malik and Ali \cite[Theorem~4.3]{malikali} proved that the complete bipartite graph $K_{m,n}$ is unlabellable when the conditions $m \ge 2$ and $n \ge 3$ both hold, and that the graph $K_n-e$ is unlabellable for all $n \ge 5$ and an arbitrary edge~$e$; these results follow by combining Observation~\ref{obs:induced} with Theorem~\ref{thm:minunlab5}~(i) and (iv), respectively. 

We finally present several claims that are stated without proof in~\cite{malikali}.
For an integer $n \ge 4$, the \emph{wheel graph} $W_n$ is formed by joining a single vertex to every vertex of a cycle graph on $n-1$ vertices.
For positive integers $m$ and $n$, the \emph{fan graph} $F_{m,n}$ is formed by joining $m$ isolated vertices to every vertex of a path on $n$ vertices.

\begin{cclaim}[{\cite[p.453]{malikali}}]
\label{claim:Weven}
The wheel graph $W_n$ is unlabellable for even $n \ge 6$.
\end{cclaim}

\begin{cclaim}[{\cite[p.453]{malikali}}]
\label{claim:fan}
\mbox{}
\begin{enumerate}[(i)]
\item[(i)]
The fan graph $F_{1,n}$ is labellable for all $n \ge 1$.
\item[(ii)]
The fan graph $F_{m,n}$ is unlabellable when the conditions $m \ge 4$ and $n \ge 2$ both hold.
\end{enumerate}
\end{cclaim}

\begin{cclaim}[{\cite[p.452]{malikali}}]
\label{claim:minunlab6} 
There are exactly four minimally unlabellable graphs on six vertices, namely:

\begin{inparaenum}[(i)]
\item \begin{tikzpicture}[scale=0.75]
\node [n, label=below: {$U_6$}] (v1) at (0,0) {};
\node [n, label=left: {$U_2$}] (v2) at (-2.5,1.5) {};
\node [n, label=below: {$U_3$}] (v3) at (-1,1.5) {};
\node [n, label=below: {$U_4$}] (v4) at (1,1.5) {};
\node [n, label=right: {$U_5$}] (v5) at (2.5,1.5) {};
\node [n, label=above: {$U_1$}] (v6) at (0,3) {};
\draw [e] (v1) -- (v2) -- (v3) -- (v4) -- (v5) -- (v6) -- (v2);
\draw [e] (v1) -- (v5);
\end{tikzpicture} \hspace{0.4in}
\item \begin{tikzpicture}[scale=1.0]
\node [n, label=below: {$U_4$}] (v1) at (0,0) {};
\node [n, label=left: {$U_2$}] (v2) at (-1,1) {};
\node [n, label=right: {$U_3$}] (v3) at (1,1) {};
\node [n, label=above: {$U_1$}] (v4) at (0,2) {};
\node [n, label=above: {$U_5$}] (v5) at (2,2) {};
\node [n, label=below: {$U_6$}] (v6) at (2,0) {};
\draw [e] (v1) -- (v2) -- (v4) -- (v3) -- (v1) -- (v6) -- (v5) -- (v4);
\draw [e] (v2) -- (v3);
\end{tikzpicture}

\vspace{1em}
\item \begin{tikzpicture}[scale=0.7]
\node [n, label=above: {$U_1$}] (v1) at (2,1.5) {};
\node [n, label=left: {$U_2$}] (v2) at (0,0) {};
\node [n, label=right: {$U_3$}] (v3) at (4,0) {};
\node [n, label={[xshift=0.4cm,yshift=-0.1cm]$U_4$}] (v4) at (2,-0.5) {};
\node [n, label=below: {$U_5$}] (v5) at (1,-2) {};
\node [n, label=below: {$U_6$}] (v6) at (3,-2) {};
\draw [e] (v1) -- (v2) -- (v5) -- (v4) -- (v1) -- (v3) -- (v6);
\draw [e] (v5) -- (v6) -- (v4);
\end{tikzpicture} \hspace{0.75in}
\item \begin{tikzpicture}[scale=0.7]
\node [n, label=above: {$ $}] (v1) at (2,1.5) {};
\node [n, label=left: {$ $}] (v2) at (0,0) {};
\node [n, label=right: {$ $}] (v3) at (4,0) {};
\node [n, label={[yshift=-1.05cm]$ $}] (v4) at (2,-0.5) {};
\node [n, label=below: {$ $}] (v5) at (1,-2) {};
\node [n, label=below: {$ $}] (v6) at (3,-2) {};
\draw [e] (v6) -- (v3) -- (v1) -- (v2) -- (v5) -- (v4);
\draw [e] (v5) -- (v6) -- (v4) -- (v1);
\draw [e] (v2) -- (v4) -- (v3);
\end{tikzpicture}
\end{inparaenum}
\end{cclaim}

\begin{cclaim}[{\cite[p.453]{malikali}}]
\label{claim:minunlab7} 
The following three graphs on seven vertices are minimally unlabellable: \\
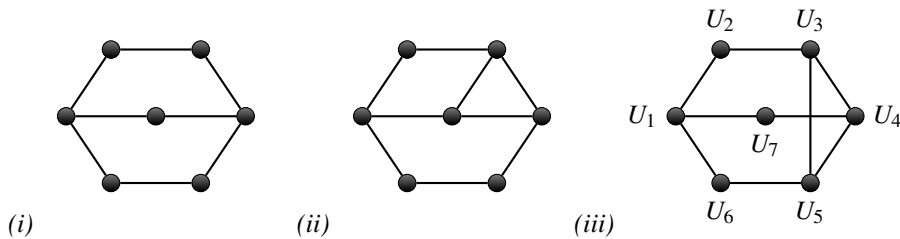
\begin{inparaenum}[(i)]
\item \begin{tikzpicture} [scale=0.59]
\node [n, label=below: {$\phantom{U_6}$}] (v1) at (0.5,0) {};
\node [n, label=below: {$ $}] (v2) at (2.5,0) {};
\node [n, label=left: {$ $}] (v3) at (-0.5,1.5) {};
\node [n, label=right: {$ $}] (v4) at (3.5,1.5) {};
\node [n, label=above: {$ $}] (v5) at (0.5,3) {};
\node [n, label=above: {$ $}] (v6) at (2.5,3) {};
\node [n, label=below: {$ $}] (v7) at (1.5,1.5) {};
\draw [e] (v1) -- (v2) -- (v4) -- (v6) -- (v5) -- (v3) -- (v1);
\draw[e] (v3) -- (v7) -- (v4);
\end{tikzpicture} \hspace{0.5em}
\item \begin{tikzpicture} [scale=0.59]
\node [n, label=below: {$\phantom{U_6}$}] (v1) at (0.5,0) {};
\node [n, label=below: {$ $}] (v2) at (2.5,0) {};
\node [n, label=left: {$ $}] (v3) at (-0.5,1.5) {};
\node [n, label=right: {$ $}] (v4) at (3.5,1.5) {};
\node [n, label=above: {$ $}] (v5) at (0.5,3) {};
\node [n, label=above: {$ $}] (v6) at (2.5,3) {};
\node [n, label=below: {$ $}] (v7) at (1.5,1.5) {};
\draw [e] (v1) -- (v2) -- (v4) -- (v6) -- (v5) -- (v3) -- (v1);
\draw[e] (v3) -- (v7) -- (v4);
\draw [e](v7) -- (v6);
\end{tikzpicture} 
\item \begin{tikzpicture} [scale=0.59]
\node [n, label=below: {$U_6$}] (v1) at (0.5,0) {};
\node [n, label=below: {$U_5$}] (v2) at (2.5,0) {};
\node [n, label=left: {$U_1$}] (v3) at (-0.5,1.5) {};
\node [n, label=right: {$U_4$}] (v4) at (3.5,1.5) {};
\node [n, label=above: {$U_2$}] (v5) at (0.5,3) {};
\node [n, label=above: {$U_3$}] (v6) at (2.5,3) {};
\node [n, label=below: {$U_7$}] (v7) at (1.5,1.5) {};
\draw [e] (v1) -- (v2) -- (v4) -- (v6) -- (v5) -- (v3) -- (v1);
\draw[e] (v3) -- (v7) -- (v4);
\draw [e](v2) -- (v6);
\end{tikzpicture}
\end{inparaenum}
\end{cclaim}

\vspace{1em}
However, parts (i) and (ii) of Claim~\ref{claim:minunlab7} do not hold because these graphs are actually labellable: 

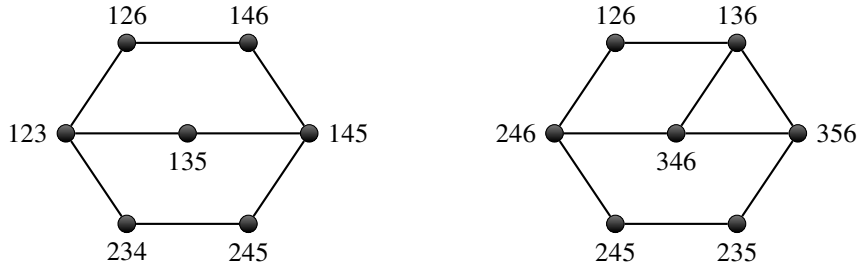
\begin{center}
\begin{tikzpicture} [scale=0.8]
\node [n, label=below: {$234$}] (v1) at (0.5,0) {};
\node [n, label=below: {$245$}] (v2) at (2.5,0) {};
\node [n, label=left: {$123$}] (v3) at (-0.5,1.5) {};
\node [n, label=right: {$145$}] (v4) at (3.5,1.5) {};
\node [n, label=above: {$126$}] (v5) at (0.5,3) {};
\node [n, label=above: {$146$}] (v6) at (2.5,3) {};
\node [n, label=below: {$135$}] (v7) at (1.5,1.5) {};
\draw [e] (v1) -- (v2) -- (v4) -- (v6) -- (v5) -- (v3) -- (v1);
\draw[e] (v3) -- (v7) -- (v4);
\end{tikzpicture} \hspace{0.5in}
\begin{tikzpicture} [scale=0.8]
\node [n, label=below: {$245$}] (v1) at (0.5,0) {};
\node [n, label=below: {$235$}] (v2) at (2.5,0) {};
\node [n, label=left: {$246$}] (v3) at (-0.5,1.5) {};
\node [n, label=right: {$356$}] (v4) at (3.5,1.5) {};
\node [n, label=above: {$126$}] (v5) at (0.5,3) {};
\node [n, label=above: {$136$}] (v6) at (2.5,3) {};
\node [n, label=below: {$346$}] (v7) at (1.5,1.5) {};
\draw [e] (v1) -- (v2) -- (v4) -- (v6) -- (v5) -- (v3) -- (v1);
\draw[e] (v3) -- (v7) -- (v4);
\draw [e](v7) -- (v6);
\end{tikzpicture} 
\end{center}

\noindent In view of this discrepancy, and to remove uncertainty as to which results have been established, we explicitly prove Claims~\ref{claim:Weven} and~\ref{claim:fan} in Section~\ref{sec:wheel-fan}, Claim~\ref{claim:minunlab6} in Section~\ref{sec:minunlab6}, and Claim~\ref{claim:minunlab7}~(iii) in Section~\ref{sec:claim7}.

\section{Induced subgraphs of a labellable graph}
\label{sec:fll}

In this section we derive a series of lemmas that constrain the form of the labelling of an induced subgraph of a labellable graph~$G$. These are useful either for determining a labelling of~$G$, or for proving that $G$ is unlabellable. 
In these lemmas, we use $123X$, for example, to mean the label set $\{1,2,3\} \cup X$ where $X$ is a (possibly empty) set disjoint from $\{1,2,3\}$. Vertices labelled as $123X$ and $257X$, for example, involve the same set $X$. We write $\ell(U)$ to mean the label of vertex~$U$.

Lemma~\ref{lem:p3} specifies the possible labellings of the path $P_3$ occurring as an induced subgraph of a labellable graph.

\begin{llemma}\label{lem:p3}
If the path $P_3$ occurs as an induced subgraph of a labellable graph, then the labelling of its vertices takes the form
\begin{center}
\begin{tikzpicture}[scale=0.3]
\node [n, label=above: {$u_1u_2T$}] (v1) at (-4,0) {};
\node [n, label={[yshift=-0.03cm]$\ell(U)$}] (v2) at (0,0) {};
\node [n, label=above: {$u'_1u'_2T$}] (v3) at (4,0) {};
\draw [e] (v1) -- (v2) -- (v3);
\end{tikzpicture}
\end{center}
for some set $T$ and distinct $u_1, u_2, u'_1$, $u'_2$. The four values of $\ell(U)$ consistent with this labelling are:

\begin{inparaenum}[(i)]
\vspace{0.3em}
\item $u_1u'_1T$; \hspace{0.5in}
\item $u_1u'_2T$; \hspace{0.5in}
\item $u'_1u_2T$; \hspace{0.5in}
\item $u_2u'_2T$.
\end{inparaenum}
\end{llemma}

\begin{proof}
Let the vertices of the induced subgraph be 
\begin{center}
\begin{tikzpicture}[scale=0.3]
\node [n, label=above: {$U_1$}] (v1) at (-4,0) {};
\node [n, label=above: {$U_2$}] (v2) at (0,0) {};
\node [n, label=above: {$U_3$}] (v3) at (4,0) {};
\draw [e] (v1) -- (v2) -- (v3);
\end{tikzpicture}
\end{center}
Since vertices $U_1$ and $U_3$ are joined by a path of length two but are not adjacent, their labels differ in exactly two elements. We may therefore write $\ell(U_1) = u_1 u_2 T$ and $\ell(U_3) = u'_1 u'_2 T$ for some set $T$ and distinct $u_1, u_2, u'_1, u'_2$. Since the label $\ell(U_2)$ must differ from each of $u_1 u_2 T$ and $u'_1 u'_2 T$ in exactly one element, this label takes one of the four values (i) to~(iv).
\qed
\end{proof}

Lemmas~\ref{lem:k3}, \ref{lem:k4e},~\ref{lem:claw} describe the possible labellings of an induced subgraph of a labellable graph, where the induced subgraph is $K_3$, $K_4-e$, $K_{1,3}$, respectively.

\begin{llemma}\label{lem:k3}
If the complete graph $K_3$ occurs as an induced subgraph of a labellable graph, then without loss of generality and for some set $X$ its labelling is exactly one of the two graphs:

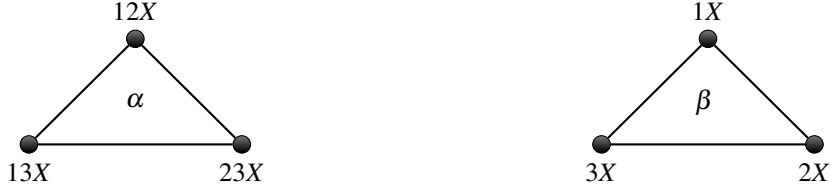
\begin{tikzpicture} [scale=0.7]
\node [n, label=above: {$12X$}] (v1) at (2,2) {};
\node [n, label=below: {$13X$}] (v2) at (0,0) {};
\node [n, label=below: {$23X$}, label={[xshift=-1.4cm, yshift=0.25cm]$\alpha$}] (v3) at (4,0) {};
\draw [e] (v1) -- (v2) -- (v3) -- (v1);
\end{tikzpicture}
\hspace{1.5in}
\begin{tikzpicture} [scale=0.7]
\node [n, label=above: {$1X$}] (v1) at (2,2) {};
\node [n, label=below: {$3X$}] (v2) at (0,0) {};
\node [n, label=below: {$2X$}, label={[xshift=-1.45cm, yshift=0.15cm]$\beta$}] (v3) at (4,0) {};
\draw [e] (v1) -- (v2) -- (v3) -- (v1);
\end{tikzpicture}
\end{llemma}

\begin{proof}
Let the vertices of the induced subgraph be 
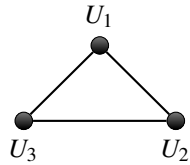
\begin{center}
\begin{tikzpicture}[scale=0.5]
\node [n, label=above: {$U_1$}] (v1) at (2,2) {};
\node [n, label=below: {$U_3$}] (v2) at (0,0) {};
\node [n, label=below: {$U_2$}] (v3) at (4,0) {};
\draw [e] (v1) -- (v2) -- (v3) -- (v1);
\end{tikzpicture}
\end{center}
\noindent and let $\ell(U_1)=\{u_1,u_2,u_3,\dots,u_k\}$ and $\ell(U_2)=\{u'_1,u_2,u_3,\dots,u_k\}$, where the variables $u'_1,u_1,u_2,u_3,\dots,u_k$ are all distinct and $k \ge 1$. Since $U_3$ is adjacent to $U_1$ and $U_2$, if $\ell(U_3)$ contains both $u_1$ and~$u'_1$ then we may take $\ell(U_3) = \{u_1,u'_1,u_3,\dots,u_k\}$ where $k \geq 2$. The resulting graph has the form $\alpha$ with $(u_1,u_2,u'_1) = (1,2,3)$ and \mbox{$\{u_3,\dots,u_k\} = X$}.

Otherwise, we may assume that $u_1 \notin \ell(U_3)$. Since $U_3$ is adjacent to~$U_1$ and distinct from $U_2$, we have $\ell(U_3) = \{u''_1,u_2,u_3,\dots,u_k\}$ where $u''_1 \notin \{u'_1,u_1,u_2,u_3,\dots,u_k\}$. The resulting graph has the form $\beta$ with $(u_1,u'_1,u''_1) = (1,2,3)$ and $\{u_2,\dots,u_k\} = X$.
\qed
\end{proof}

\begin{llemma}\label{lem:k4e}
If $K_4-e$ occurs as an induced subgraph of a labellable graph (where $e$ is an arbitrary edge of $K_4$), then without loss of generality and for some set $X$ its labelling is
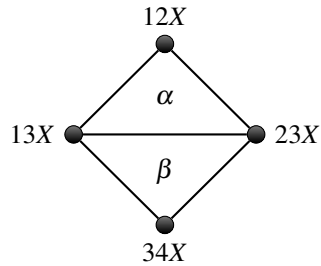
\begin{center}
\begin{tikzpicture} [scale=0.6]
\node [n, label=above: {$12X$}] (v1) at (2,2) {};
\node [n, label=left: {$13X$}] (v2) at (0,0) {};
\node [n, label=right: {$23X$}, label={[xshift=-1.2cm, yshift=0.15cm]$\alpha$}, label={[xshift=-1.2cm, yshift=-0.9cm]$\beta$}] (v3) at (4,0) {};
\node [n, label=below: {$34X$}] (v4) at (2,-2) {};
\draw [e] (v1) -- (v2) -- (v3) -- (v1);
\draw [e] (v2) -- (v4) -- (v3);
\end{tikzpicture}
\end{center}
in which the types $\alpha$, $\beta$ of the two induced $K_3$ subgraphs are as shown in~Lemma~\ref{lem:k3}.
\end{llemma}

\begin{proof}
Let the vertices of the induced subgraph be
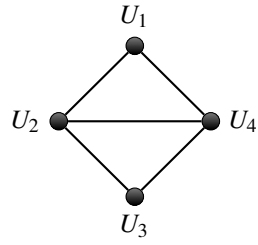
\begin{center}
\begin{tikzpicture}[scale=0.5]
\node [n, label=above: {$U_1$}] (v1) at (2,2) {};
\node [n, label=left: {$U_2$}] (v2) at (0,0) {};
\node [n, label=right: {$U_4$}] (v3) at (4,0) {};
\node [n, label=below: {$U_3$}] (v4) at (2,-2) {};
\draw [e] (v2) -- (v1) -- (v3) -- (v2) -- (v4) -- (v3);
\end{tikzpicture}
\end{center}
By Lemma~\ref{lem:p3}, we may take 
$\ell(U_1) = \{u_1,u_2,u_3,\dots,u_k\}$, 
$\ell(U_2) = \{u'_1,u_2,u_3,\dots,u_k\}$, 
$\ell(U_3) = \{u'_1,u'_2,u_3,\dots,u_k\}$, 
where $u'_1,u'_2,u_1,u_2,u_3\dots,u_k$ are all distinct and $k \geq 2$.
Since $U_4$ is adjacent to both $U_1$ and $U_3$, its possible labellings are determined by Lemma~\ref{lem:p3} with $T = \{u_3,\dots,u_k\}$; because $U_4$ is distinct from and adjacent to $U_2$, only cases (i) and (iv) of Lemma~\ref{lem:p3} can occur. 

In case (i), we have $\ell(U_4)=\{u_1,u'_1,u_3,\dots,u_k\}$. The resulting graph has the given form with $(u_1,u_2,u'_1,u'_2) = (2,1,3,4)$ and $\{u_3,\dots,u_k\} = X$.

In case (iv), we have $\ell(U_4)=\{u_2,u'_2,u_3,\dots,u_k\}$. The resulting graph (after reflection through a horizontal axis) has the given form with $(u_1,u_2,u'_1,u'_2) = (4,3,1,2)$ and $\{u_3,\dots,u_k\} = X$.

In both cases, the types $\alpha$, $\beta$ of the two induced $K_3$ subgraphs are as depicted in~Lemma~\ref{lem:k3}.
\qed
\end{proof}

\begin{llemma}\label{lem:claw}
If the complete bipartite graph $K_{1,3}$ occurs as an induced subgraph of a labellable graph, then without loss of generality and for some set $X$ its labelling is
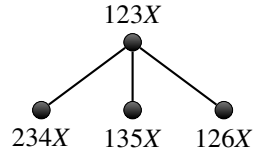
\begin{center}
\begin{tikzpicture}[scale=0.3]
\node [n, label=above: {$123X$}] (v1) at (3,3) {};
\node [n, label=below: {$234X$}] (v2) at (-1,0) {};
\node [n, label=below: {$135X$}] (v3) at (3,0) {};
\node [n, label=below: {$126X$}] (v4) at (7,0) {};
\draw [e] (v2) -- (v1) -- (v3);
\draw [e] (v1) -- (v4);
\end{tikzpicture}
\end{center}
\end{llemma}

\begin{proof}
Let the vertices of the induced subgraph be
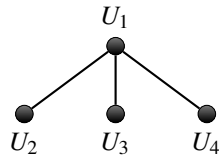
\begin{center}
\begin{tikzpicture}[scale=0.3]
\node [n, label=above: {$U_1$}] (v1) at (3,3) {};
\node [n, label=below: {$U_2$}] (v2) at (-1,0) {};
\node [n, label=below: {$U_3$}] (v3) at (3,0) {};
\node [n, label=below: {$U_4$}] (v4) at (7,0) {};
\draw [e] (v2) -- (v1) -- (v3);
\draw [e] (v1) -- (v4);
\end{tikzpicture}
\end{center}
By Lemma~\ref{lem:p3}, we may take $\ell(U_2) = \{u_1,u_2,u_3,\dots,u_k\}$, $\ell(U_1) = \{u'_1,u_2,u_3,\dots,u_k\}$, $\ell(U_3) = \{u'_1,u'_2,u_3,\dots,u_k\}$, where $u'_1,u'_2,u_1,u_2,u_3,\dots,u_k$ are all distinct and $k \geq 2$.
Since $U_4$ is adjacent to $U_1$ but not to $U_2$ and not to $U_3$, we may take $\ell(U_4)=\{u'_1,u_2,u'_3,u_4,\dots,u_k\}$ where $u'_3 \notin \{u'_1,u'_2,u_1,u_2,u_3,u_4\dots,u_k\}$ and $k \geq 3$. The resulting graph has the given form with $(u_1,u_2,u_3,u'_1,u'_2,u'_3) = (4,2,3,1,5,6)$ and with $\{u_4,\dots,u_k\} = X$.
\qed
\end{proof}

\section{Wheel graphs and fan graphs}
\label{sec:wheel-fan}

In this section we firstly determine the values of $n \ge 4$ for which the wheel graph $W_n$ is labellable, and show that for all other $n$ it is minimally unlabellable. 
We then determine the pairs $(m,n)$ for which the fan graph $F_{m,n}$ is labellable, minimally unlabellable, and unlabellable (not minimally).

The proof of Theorem~\ref{thm:wheel}~(i) is illustrated for $W_9$ in Figure~\ref{fig:oddwheellabel}. 
\begin{theorem} 
\label{thm:wheel}
\mbox{}
\begin{enumerate}
\item[(i)]
The wheel graph $W_n$ is labellable for $n=4$ and for odd $n \ge 5$.

\item[(ii)]
The wheel graph $W_n$ is minimally unlabellable for even $n \ge 6$.
\end{enumerate}
\end{theorem}

\begin{proof}
\mbox{}
\begin{enumerate}
\item[(i)]
$W_4$ is isomorphic to $K_4$, which is labellable by Theorem~\ref{thm:labfams}~(i). 

To show that $W_{2m+1}$ is labellable for each integer $m \ge 2$,
form $W_{2m+1}$ by joining a vertex~$v$ to every vertex of a cycle on vertices $v_1, v_2, \dots, v_{2m}$.
Assign the label $[m] := \{1,2,\dots,m\}$ to vertex $v$ and, for $1 \le i \le m$, 
assign the label $[m] \setminus \{i\} \cup \{m+i\}$ to $v_{2i-1}$ 
and the label $[m] \setminus \{i\} \cup \{m+1+(i \bmod m)\}$ to $v_{2i}$.

\item[(ii)]
Let $m \ge 3$ and suppose, for a contradiction, that $W_{2m}$ is labellable. 
Form $W_{2m}$ by joining a vertex~$v$ to every vertex of a cycle on vertices $v_1, v_2, \dots, v_{2m-1}$.
Each of the $2m-1$ triples of vertices
\begin{equation}\label{eqn:cyclic-seq}
\{v,v_1,v_2\}, \{v,v_2,v_3\},\dots, \{v,v_{2m-2},v_{2m-1}\}, \{v,v_{2m-1},v_1\}
\end{equation}
then induces a subgraph $K_3$ in $W_{2m}$, and the labelling of each of these induced subgraphs has exactly one of the two types $\alpha$, $\beta$ specified in Lemma~\ref{lem:k3}.
Since $m \ge 3$, the vertices of two adjacent triples in the list~\eqref{eqn:cyclic-seq} (viewed as a cyclic sequence) induce a subgraph of the form $K_4-e$, and moreover by Lemma~\ref{lem:k4e} each such induced subgraph $K_4-e$ comprises one induced subgraph $K_3$ of type $\alpha$ and one of type~$\beta$. 
Therefore the types of the induced subgraphs $K_3$ resulting from the cyclic sequence of $2m-1$ triples \eqref{eqn:cyclic-seq} alternate between $\alpha$ and $\beta$, which is a contradiction.

We conclude that $W_{2m}$ is unlabellable. To show that $W_{2m}$ is minimally unlabellable, by Observation~\ref{obs:induced} it is sufficient to show that all subgraphs obtained by removing a single vertex of $W_{2m}$ are labellable. 
The graph $W_{2m}-v$ is the cycle graph $C_{2m-1}$, which is labellable by Theorem~\ref{thm:labfams}~(ii); the graph $W_{2m}-v_i$ is an induced subgraph of $W_{2m+1}$, and so is labellable by applying Observation~\ref{obs:induced} to the result of part~(i). 
\qed
\end{enumerate}
\end{proof}

\begin{figure}
\begin{center}
\begin{tikzpicture}[scale=1.0]
\node [n, label={[xshift=-0.2cm]$v_1 = \{2,3,4,5\}$}] (v1) at (-0.9,2) {};
\node [n, label={[xshift=0.2cm]$v_2 = \{2,3,4,6\}$}] (v2) at (0.9,2) {};
\node [n, label=right: {$v_3 = \{1,3,4,6\}$}] (v3) at (2,0.9) {};
\node [n, label=right: {$v_4 = \{1,3,4,7\}$}] (v4) at (2,-0.9) {};
\node [n, label={[xshift=0.2cm,yshift=-0.75cm]$v_5 = \{1,2,4,7\}$}] (v5) at (0.9,-2) {};
\node [n, label={[xshift=-0.2cm,yshift=-0.75cm]$v_6 = \{1,2,4,8\}$}] (v6) at (-0.9,-2) {};
\node [n, label=left: {$v_7 = \{1,2,3,8\}$}] (v7) at (-2,-0.9) {};
\node [n, label=left: {$v_8 = \{1,2,3,5\}$}] (v8) at (-2,0.9) {};
\node [n, label={[xshift=1.15cm,yshift=-0.35cm]$v = \{1,2,3,4\}$}] (v9) at (0,0) {};
\draw [e] (v1) -- (v2) -- (v3) -- (v4) -- (v5) -- (v6) -- (v7) -- (v8) -- (v1) -- (v9) -- (v2);
\draw [e] (v3) -- (v9) -- (v4);
\draw [e] (v5) -- (v9) -- (v6);
\draw [e] (v7) -- (v9) -- (v8);
\end{tikzpicture}
\caption{A labelling of the wheel graph $W_9$ according to the proof of Theorem~\ref{thm:wheel}~(i).}
\label{fig:oddwheellabel}
\end{center}
\end{figure}
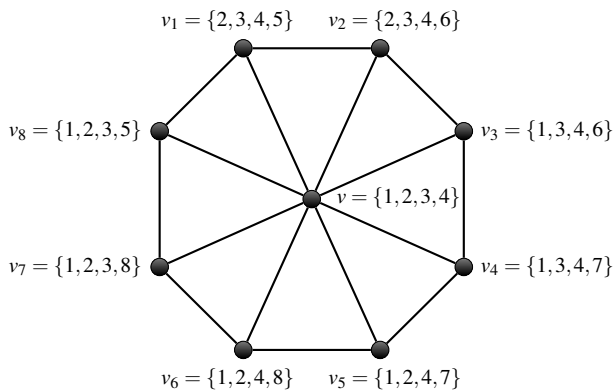

Theorem~\ref{thm:wheel}~(ii) provides an infinite family of minimally unlabellable graphs. In particular, it establishes Claim~\ref{claim:Weven}.

\begin{theorem} 
\label{thm:fan}
\mbox{}
\begin{enumerate}
\item[(i)]
The fan graphs $F_{2,2}$, $F_{2,3}$, $F_{m,1}$ and $F_{1,n}$ are labellable for all $m, n \ge 1$.
\item[(ii)]
The fan graph $F_{3,2}$ is minimally unlabellable.
\item[(iii)]
The fan graph $F_{m,n}$ is unlabellable (not minimally) for all $(m,n)$ not specified in (i) and (ii).
\end{enumerate}
\end{theorem}

\begin{proof}
\begin{enumerate}
\item[(i)]
The graphs $F_{2,2}$ and $F_{2,3}$ are labellable:
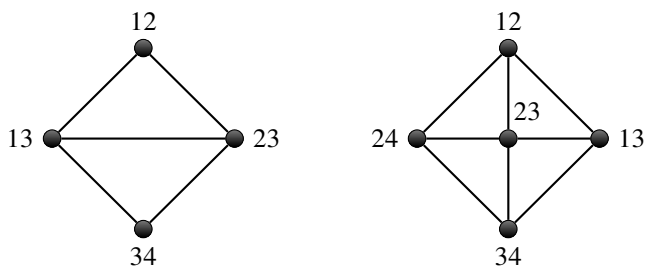
\begin{center}
\begin{tikzpicture} [scale=0.6]
\node [n, label=above: {$12$}] (v1) at (2,2) {};
\node [n, label=left: {$13$}] (v2) at (0,0) {};
\node [n, label=right: {$23$}] (v3) at (4,0) {};
\node [n, label=below: {$34$}] (v4) at (2,-2) {};
\draw [e] (v1) -- (v2) -- (v3) -- (v1);
\draw [e] (v2) -- (v4) -- (v3);

\node [n, label=above: {$12$}] (v12) at (10,2) {};
\node [n, label=left: {$24$}] (v24) at (8,0) {};
\node [n, label={[xshift=0.25cm,yshift=0cm]$23$}] (v23) at (10,0) {};
\node [n, label=right: {$13$}] (v13) at (12,0) {};
\node [n, label=below: {$34$}] (v34) at (10,-2) {};
\draw [e] (v24) -- (v12) -- (v13) -- (v34) -- (v24) -- (v23) -- (v13);
\draw [e] (v12) -- (v23) -- (v34);
\end{tikzpicture}
\end{center}

We next show that $F_{m,1}$ is labellable for all $m \ge 1$.
Form $F_{m,1}$ by joining $m$ isolated vertices $v_1, v_2, \dots, v_m$ to a single vertex $v$.
Assign the label $[m] := \{1,2,\dots,m\}$ to vertex $v$ and, for $1 \le i \le m$, 
assign the label $[m] \setminus \{i\} \cup \{m+i\}$ to $v_i$.

It remains to show that $F_{1,n}$ is labellable for all $n \ge 2$.
Since $F_{1,n}$ is an induced subgraph of $W_{2n+1}$ for $n \ge 2$, this follows by applying Observation~\ref{obs:induced} to Theorem~\ref{thm:wheel}~(i).

\item[(ii)]
The graph $F_{3,2}$ is the graph (ii) of Theorem~\ref{thm:minunlab5}, which is minimally unlabellable.

\item[(iii)]
By Observation~\ref{obs:induced}, the graph $F_{2,4}$ is unlabellable (not minimally) because it contains graph (iii) of Theorem~\ref{thm:minunlab5} as a proper induced subgraph:

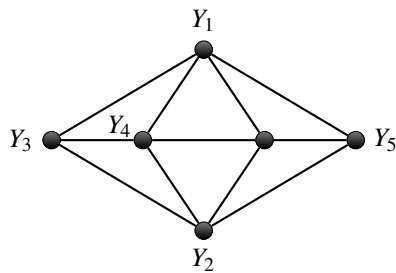
\begin{center}
\begin{tikzpicture}[scale=0.8]
\node [n, label=below: {$Y_2$}] (v1) at (0,0) {};
\node [n, label=left: {$Y_3$}] (v2) at (-2.5,1.5) {};
\node [n, label= {[xshift=-0.3cm, yshift=-0.2cm]$Y_4$}] (v3) at (-1,1.5) {};
\node [n, label={$ $}] (v4) at (1,1.5) {};
\node [n, label=right: {$Y_5$}] (v5) at (2.5,1.5) {};
\node [n, label=above: {$Y_1$}] (v6) at (0,3) {};
\draw [e] (v1) -- (v5) -- (v6) -- (v2) -- (v1) -- (v3) -- (v6) -- (v4) -- (v1);
\draw [e] (v2) -- (v3) -- (v4) -- (v5);
\end{tikzpicture}
\end{center}

The result follows from the observation that the fan graph $F_{m,n}$ is a proper induced subgraph of $F_{m+1,n}$ and of $F_{m,n+1}$.
\qed
\end{enumerate}
\end{proof}

In particular, Theorem~\ref{thm:fan} establishes Claim~\ref{claim:fan}.

\section{Minimally unlabellable graphs on at most five vertices} 
\label{sec:minunlab5}

Theorem~\ref{thm:minunlab5} specifies that there are exactly four minimally unlabellable graphs on at most five vertices.
In this section we use the results of Section~\ref{sec:fll} to verify briefly that these four graphs are indeed unlabellable. 
Appendix~A demonstrates explicitly that all the other 27 connected graphs on at most five vertices are labellable, which by Theorem~\ref{thm:genconstr}~(i) then implies Theorem~\ref{thm:minunlab5}.

Consider each of the graphs (i) to (iv) in Theorem~\ref{thm:minunlab5} in turn, and suppose for a contradiction that the graph is labellable. 
\begin{enumerate}[(i)]
\item By Lemma~\ref{lem:claw} applied to the subgraph induced by vertices $U_1, U_3, U_4$, $U_5$, we may assign labels
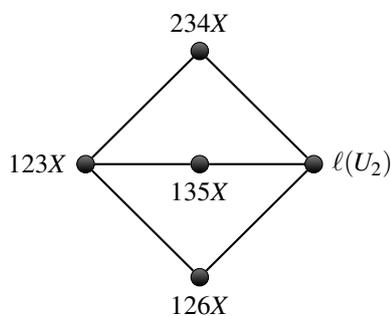
\begin{center}
\begin{tikzpicture}[scale=1.0]
\node [n, label=below: {$126X$}] (v1) at (0,0) {};
\node [n, label=left: {$123X$}] (v2) at (-1.5,1.5) {};
\node [n, label=below: {$135X$}] (v3) at (0,1.5) {};
\node [n, label=right: {$\ell(U_2)$}] (v4) at (1.5,1.5) {};
\node [n, label=above: {$234X$}] (v5) at (0,3) {};
\draw [e] (v1) -- (v2) -- (v5) -- (v4) -- (v1);
\draw [e] (v2) -- (v3) -- (v4);
\end{tikzpicture}
\end{center}
for some set~$X$. Apply Lemma~\ref{lem:p3} with $T= \{3\} \cup X$ to the induced path $P_3$ on the vertices lablled $234X$, $\ell(U_2)$, $135X$.
None of the cases (i) to (iv) of Lemma~\ref{lem:p3} is consistent with the condition that $\ell(U_2)$ should be distinct from $123X$ and differ from $126X$ in exactly one element.

\item By Lemma~\ref{lem:k4e} applied to the subgraph induced by vertices $V_1, V_2, V_3$, $V_4$, we may assign labels
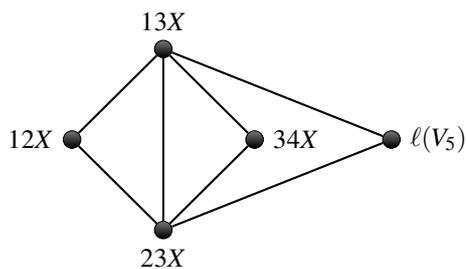
\begin{center}
\begin{tikzpicture}[scale=1.2]
\node [n, label=below: {$23X$}] (v1) at (0,0) {};
\node [n, label=left: {$12X$}] (v2) at (-1,1) {};
\node [n, label=right: {$34X$}] (v3) at (1,1) {};
\node [n, label=above: {$13X$}] (v4) at (0,2) {};
\node [n, label=right: {$\ell(V_5)$}] (v5) at (2.5,1) {};
\draw [e] (v1) -- (v2) -- (v4) -- (v3) -- (v1) -- (v4) -- (v5) -- (v1);
\end{tikzpicture}
\end{center}
for some set $X$. 
Apply Lemma~\ref{lem:k3} to the induced subgraph $K_3$ on the vertices labelled $13X$, $\ell(V_5)$, $23X$. Neither of the outcomes in Lemma~\ref{lem:k3} is consistent with the condition that $\ell(V_5)$ should be distinct from $12X$ and differ from $34X$ in more than one element.

\item By Lemma~\ref{lem:k4e} applied to the subgraph induced by vertices $Y_1, Y_2, Y_3$, $Y_4$, we may assign labels
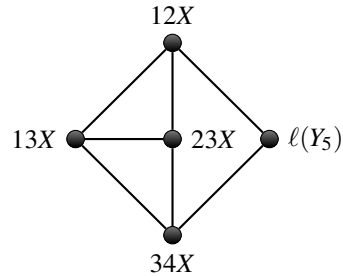
\begin{center}
\begin{tikzpicture}[scale=0.85]
\node [n, label=below: {$34X$}] (v1) at (0,0) {};
\node [n, label=left: {$13X$}] (v2) at (-1.5,1.5) {};
\node [n, label=right: {$23X$}] (v3) at (0,1.5) {};
\node [n, label=right: {$\ell(Y_5)$}] (v4) at (1.5,1.5) {};
\node [n, label=above: {$12X$}] (v5) at (0,3) {};
\draw [e] (v1) -- (v3) -- (v5) -- (v4) -- (v1) -- (v2) -- (v5);
\draw [e] (v2) -- (v3);
\end{tikzpicture}
\end{center}
for some set $X$. 
Apply Lemma~\ref{lem:p3} to the induced path $P_3$ on the vertices labelled $12X$, $\ell(Y_5)$, $34X$.
None of the cases (i) to (iv) of Lemma~\ref{lem:p3} is consistent with the condition that $\ell(Y_5)$ should differ from both $13X$ and~$23X$ in more than one element.

\item By Lemma~\ref{lem:k4e} applied to the subgraph induced by vertices $Z_1, Z_2, Z_3$, $Z_5$, may assign labels
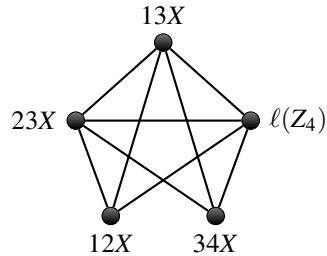
\begin{center}
\begin{tikzpicture}[scale=1.15]
\node [n, label=left: {$23X$}] (v1) at (-1,0.1) {};
\node [n, label=right: {$\ell(Z_4)$}] (v2) at (1,0.1) {};
\node [n, label=above: {$13X$}] (v3) at (0,1) {};
\node [n, label=below: {$12X$}] (v4) at (-0.6,-1) {};
\node [n, label=below: {$34X$}] (v5) at (0.6,-1) {};
\draw [e] (v5) -- (v2) -- (v4) -- (v3) -- (v5) -- (v1) -- (v2) -- (v3) -- (v1) --(v4);
\end{tikzpicture}
\end{center}
for some set $X$. 
Apply Lemma~\ref{lem:p3} to the induced path $P_3$ on the vertices labelled $12X$, $\ell(Z_4)$, $34X$.
None of the cases (i) to (iv) of Lemma~\ref{lem:p3} is consistent with the condition that $\ell(Z_4)$ should differ from both $13X$ and~$23X$ in exactly one element.
\end{enumerate}

\section{Minimally unlabellable graphs on six vertices} 
\label{sec:minunlab6}

Claim~\ref{claim:minunlab6} states that there are exactly four minimally unlabellable graphs on six vertices.
In this section we use the results of Section~\ref{sec:fll} to prove that these four graphs are indeed unlabellable. 
It follows that these graphs are minimally unlabellable: each of their proper induced subgraphs is labellable, by Theorem~\ref{thm:minunlab5}.
Appendix~B demonstrates explicitly that of the other 108 connected graphs on six vertices,
69 are labellable and 39 contain as a proper induced subgraph some unlabellable five-vertex graph specified in Theorem~\ref{thm:minunlab5}. Together with Theorem~\ref{thm:genconstr}~(i), this proves Claim~\ref{claim:minunlab6}.

Graph (iv) in Claim~\ref{claim:minunlab6} is $W_6$, which is unlabellable by Theorem~\ref{thm:wheel}~(ii).
Consider each of the other graphs (i) to (iii) in Claim~\ref{claim:minunlab6} in turn, and suppose for a contradiction that the graph is labellable. 
\begin{enumerate}[(i)]
\item By Lemma~\ref{lem:claw} applied to $U_1, U_2, U_3$, $U_6$, we may assign labels
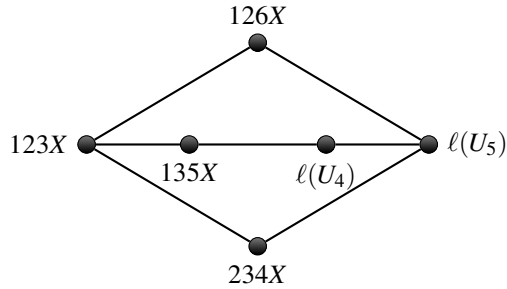
\begin{center}
\begin{tikzpicture}[scale=0.9]
\node [n, label=below: {$234X$}] (v1) at (0,0) {};
\node [n, label=left: {$123X$}] (v2) at (-2.5,1.5) {};
\node [n, label=below: {$135X$}] (v3) at (-1,1.5) {};
\node [n, label=below: {$\ell(U_4)$}] (v4) at (1,1.5) {};
\node [n, label=right: {$\ell(U_5)$}] (v5) at (2.5,1.5) {};
\node [n, label=above: {$126X$}] (v6) at (0,3) {};
\draw [e] (v1) -- (v2) -- (v3) -- (v4) -- (v5) -- (v6) -- (v2);
\draw [e] (v1) -- (v5);
\end{tikzpicture}
\end{center}
for some set $X$. Apply Lemma~\ref{lem:p3} with $T=\{2\} \cup X$ to the induced path $P_3$ on the vertices labelled $126X$, $\ell(U_5)$, $234X$.
The only case of Lemma~\ref{lem:p3} that is consistent with the condition that $\ell(U_5)$ should differ from $123X$ in more than one element occurs when $\ell(U_5) = 246X$. But then the vertices labelled $135X$ and $246X$ are joined by a path of length two but their labels differ in three elements, which is a contradiction.

\item By Lemma~\ref{lem:k4e} applied to $U_1, U_2, U_3$, $U_4$, we may assign labels
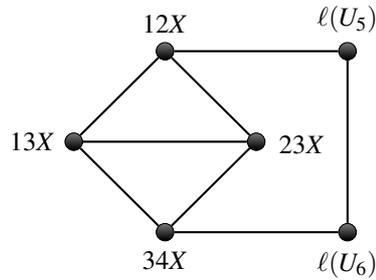
\begin{center}
\begin{tikzpicture}[scale=1.2]
\node [n, label=below: {$34X$}] (v1) at (0,0) {};
\node [n, label=left: {$13X$}] (v2) at (-1,1) {};
\node [n, label={[xshift=0.6cm, yshift=-0.4cm]$23X$}] (v3) at (1,1) {};
\node [n, label=above: {$12X$}] (v4) at (0,2) {};
\node [n, label=above: {$\ell(U_5)$}] (v5) at (2,2) {};
\node [n, label=below: {$\ell(U_6)$}] (v6) at (2,0) {};
\draw [e] (v1) -- (v2) -- (v4) -- (v3) -- (v1) -- (v6) -- (v5) -- (v4);
\draw [e] (v2) -- (v3);
\end{tikzpicture}
\end{center}
for some set $X$. By applying Lemma~\ref{lem:p3} to the induced path $P_3$ on the vertices labelled $23X$, $34X$, $\ell(U_6)$, we may take $\ell(U_6)=45X$. 
Then apply Lemma~\ref{lem:p3} to the induced path $P_3$ on the vertices labelled $12X$, $\ell(U_5)$, $45X$. 
None of the cases (i) to (iv) of Lemma~\ref{lem:p3} is consistent with the condition that $\ell(U_5)$ should differ from both $13X$ and~$23X$ in more than one element.

\item By Lemma~\ref{lem:claw} applied to $U_1, U_2, U_3$, $U_4$, we may assign labels
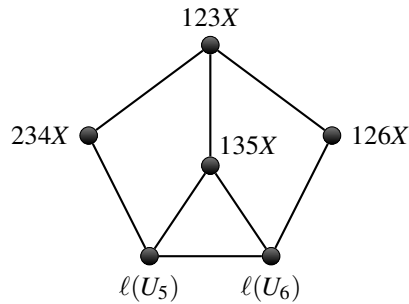
\begin{center}
\begin{tikzpicture}[scale=0.8]
\node [n, label=above: {$123X$}] (v1) at (2,1.5) {};
\node [n, label=left: {$234X$}] (v2) at (0,0) {};
\node [n, label=right: {$126X$}] (v3) at (4,0) {};
\node [n, label={[xshift=0.5cm,yshift=-0.1cm]$135X$}] (v4) at (2,-0.5) {};
\node [n, label=below: {$\ell(U_5)$}] (v5) at (1,-2) {};
\node [n, label=below: {$\ell(U_6)$}] (v6) at (3,-2) {};
\draw [e] (v1) -- (v2) -- (v5) -- (v4) -- (v1) -- (v3) -- (v6);
\draw [e] (v5) -- (v6) -- (v4);
\end{tikzpicture}
\end{center}
for some set $X$. Apply Lemma~\ref{lem:p3} with $T = \{3\} \cup X$ to the induced path $P_3$ on the vertices labelled $234X$, $\ell(U_5)$, $135X$. 
The only case of Lemma~\ref{lem:p3} that is consistent with the condition that $\ell(U_5)$ should differ from $123X$ in more than one element occurs when $\ell(U_5) = 345X$. But then the vertices labelled $126X$ and $345X$ are joined by a path of length two but their labels differ in three elements, which is a contradiction.
\end{enumerate}

\section{Proof of Claim~\ref{claim:minunlab7}~(iii)} 
\label{sec:claim7}
In this section we prove that graph (iii) in Claim~\ref{claim:minunlab7} is unlabellable. It follows that this graph is minimally unlabellable, as claimed, because each of its proper induced subgraphs is labellable by Theorem~\ref{thm:minunlab5} and Claim~\ref{claim:minunlab6} (which was established in Section~\ref{sec:minunlab6}). 

Suppose, for a contradiction, that graph (iii) in Claim~\ref{claim:minunlab7} is labellable.
Then by Lemma~\ref{lem:claw} applied to $U_1, U_2, U_6$, $U_7$, we may assign labels
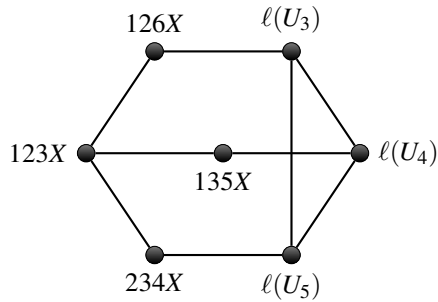
\begin{center}
\begin{tikzpicture} [scale=0.9]
\node [n, label=below: {$234X$}] (v1) at (0.5,0) {};
\node [n, label=below: {$\ell(U_5)$}] (v2) at (2.5,0) {};
\node [n, label=left: {$123X$}] (v3) at (-0.5,1.5) {};
\node [n, label=right: {$\ell(U_4)$}] (v4) at (3.5,1.5) {};
\node [n, label=above: {$126X$}] (v5) at (0.5,3) {};
\node [n, label=above: {$\ell(U_3)$}] (v6) at (2.5,3) {};
\node [n, label=below: {$135X$}] (v7) at (1.5,1.5) {};
\draw [e] (v1) -- (v2) -- (v4) -- (v6) -- (v5) -- (v3) -- (v1);
\draw[e] (v3) -- (v7) -- (v4);
\draw [e](v2) -- (v6);
\end{tikzpicture}
\end{center}
for some set $X$. 
The label $\ell(U_4)$ differs from $135X$ in exactly one element. This element cannot belong to $X$, because $\ell(U_4)$ must differ from $126X$ in exactly two elements. Therefore $\ell(U_4)$ contains $13X$ or $15X$ or $35X$; the first possibility is excluded because $\ell(U_4)$ must differ from $123X$ in exactly two elements, and for the same reason $\ell(U_4)$ does not contain $2$. So we may take $\ell(U_4)$ to be the union of one element of $\{4,6,7\}$ with either $15X$ or~$35X$. We need to consider only the union with $15X$ because the mapping that interchanges $1$ with $3$, and $4$ with $6$, maps the partially labelled graph to (a reflection through a horizontal axis of) itself. 
This leaves $\ell(U_4)$ as one of $145X$, $156X$, $157X$. 
The only one of these possibilities that is consistent with the condition that $\ell(U_4)$ should differ from $234X$ in exactly two elements is $\ell(U_4) = 145X$.

Now apply Lemma~\ref{lem:p3} with $T=\{1\} \cup X$ to the induced path $P_3$ on the vertices labelled $126X$, $\ell(U_3)$, $145X$. The only case of Lemma~\ref{lem:p3} that is consistent with the condition that $\ell(U_3)$ should differ from both $123X$ and $135X$ in more than one element is $\ell(U_3) = 146X$.
Then apply Lemma~\ref{lem:p3} with $T=\{4\} \cup X$ to the induced path $P_3$ on the vertices labelled $234X$, $\ell(U_5)$, $145X$. The only case of Lemma~\ref{lem:p3} that is consistent with the condition that $\ell(U_5)$ should differ from both $123X$ and $135X$ in more than one element is $\ell(U_5) = 245X$. But then $\ell(U_3)$ and $\ell(U_5)$ differ in more than one element, which is a contradiction.

\section{Conclusion} 
\label{sec:conc}

We have extended the definition of the $\gamma$-graph $\gamma\cdot G$ from distance\-/1\-/domination to distance\-/$d$\-/domination, and have shown in Corollary~\ref{cor:realisegammaddotg} that the existence of such a generalised $\gamma$-graph $H$ depends only on whether $H$ is labellable.

We have completely determined the wheel graphs and fan graphs that are labellable.
We have verified for graphs on at most five vertices, and established for graphs on six vertices, precisely which graphs are minimally unlabellable. We have also given an explicit labelling of all connected labellable graphs on at most six vertices. A similar classification procedure could in principle be applied to the $853$ connected graphs on seven vertices, and even the $11117$ connected graphs on eight vertices~\cite{oeisconnected}, although the procedure should be automated as much as possible to avoid errors.

We have exhibited an infinite family of minimally unlabellable graphs in Theorem~\ref{thm:wheel}~(ii).
One might hope to uncover further such families by examining the minimally unlabellable graphs on at most six (and, in future, seven or eight) vertices.
At first sight, the form of graph (i) in Theorem~\ref{thm:minunlab5} and graph (i) in Claim~\ref{claim:minunlab6} suggests such a family, but the next member of this presumed family is in fact labellable (as are all subsequent members):

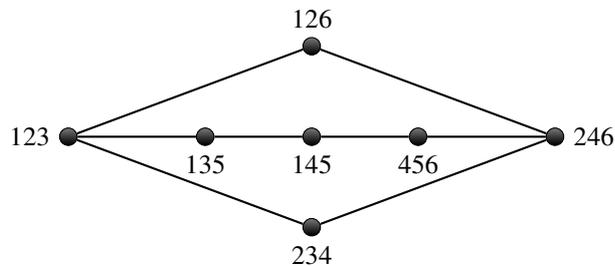
\begin{center}
\begin{tikzpicture}[scale=0.8]
\node [n, label=below: {$234$}] (v1) at (6,0) {};
\node [n, label=left: {$123$}] (v2) at (2,1.5) {};
\node [n, label=below: {$135$}] (v3) at (4.25,1.5) {};
\node [n, label=below: {$145$}] (v4) at (6,1.5) {};
\node [n, label=below: {$456$}] (v5) at (7.75,1.5) {};
\node [n, label=right: {$246$}] (v6) at (10,1.5) {};
\node [n, label=above: {$126$}] (v7) at (6,3) {};
\draw [e] (v2) -- (v3) -- (v4) -- (v5) -- (v6) -- (v7) -- (v2) -- (v1) -- (v6);
\end{tikzpicture}
\end{center}

\begin{acknowledgements}
We are grateful to Ladislav Stacho for his helpful suggestions for improving this paper.
We thank the anonymous referee for constructive comments, especially those simplifying the proof of Theorem~\ref{thm:realisegammaddotg}.
\end{acknowledgements}

\section*{Conflict of interest}
The authors declare that they have no conflict of interest.

%\bibliography{gammagraphs}

\begin{thebibliography}{10}
\providecommand{\url}[1]{{#1}}
\providecommand{\urlprefix}{URL }
\expandafter\ifx\csname urlstyle\endcsname\relax
  \providecommand{\doi}[1]{DOI~\discretionary{}{}{}#1}\else
  \providecommand{\doi}{DOI~\discretionary{}{}{}\begingroup
  \urlstyle{rm}\Url}\fi

\bibitem{lakshmanan}
{Aparna Lakshmanan}, S., Vijayakumar, A.: The gamma graph of a graph.
\newblock AKCE International Journal of Graphs and Combinatorics \textbf{7},
  53--59 (2010)

\bibitem{berge}
Berge, C.: The Theory of Graphs and its Applications.
\newblock John Wiley \& Sons Inc., New York (1962)

\bibitem{bien}
Bie\'n, A.: Gamma graphs of some special classes of trees.
\newblock Annales Mathematicae Silesianae \textbf{29}, 25--34 (2015)

\bibitem{brouwercohenneumaier}
Brouwer, A., Cohen, A., Neumaier, A.: Distance-{R}egular {G}raphs.
\newblock Springer-Verlag, Berlin (1989)

\bibitem{connelly}
Connelly, E., Hutson, K., Hedetniemi, S.: A note on {$\gamma$}-graphs.
\newblock AKCE International Journal of Graphs and Combinatorics \textbf{8},
  23--31 (2011)

\bibitem{cvetpet}
Cvetkovi\'c, D., Petri\'c, M.: A table of connected graphs on six vertices.
\newblock Discrete Mathematics \textbf{50}, 37--49 (1984)

\bibitem{dyck-masters}
Dyck, A.: The realisability of $\gamma$-graphs.
\newblock Master's thesis, Simon Fraser University (2017).
\newblock Available at \url{http://summit.sfu.ca/item/17513}

\bibitem{edmonds-fulkerson}
Edmonds, J., Fulkerson, D.R.: Bottleneck extrema.
\newblock Journal of Combinatorial Theory \textbf{8}, 299--306 (1970)

\bibitem{edwards}
Edwards, M.: Vertex-criticality and bicriticality for independent domination
  and total domination in graphs.
\newblock Ph.D. thesis, University of Victoria (2015)

\bibitem{fricke}
Fricke, G., Hedetniemi, S., Hedetniemi, S., Hutson, K.: {$\gamma$}-graphs of
  graphs.
\newblock Discussiones Mathematicae Graph Theory \textbf{31}, 517--531 (2011)

\bibitem{handbook}
Gross, J., Yellen, J., Zhang, P. (eds.): Handbook of {G}raph {T}heory, second
  edn.
\newblock Discrete Mathematics and its Applications (Boca Raton). CRC Press,
  Boca Raton, FL (2014)

\bibitem{fundamentalsadvanced}
Haynes, T., Hedetniemi, S., Slater, P.: Domination in Graphs: Advanced Topics.
\newblock Marcel Dekker, Inc., New York (1998)

\bibitem{fundamentals}
Haynes, T., Hedetniemi, S., Slater, P.: Fundamentals of Domination in Graphs.
\newblock Marcel Dekker, Inc., New York (1998)

\bibitem{henning}
Henning, M., Lichiardopol, N.: Distance domination in graphs with given minimum
  and maximum degree.
\newblock J. Comb. Optim. \textbf{34}, 545--553 (2017)

\bibitem{hhldom}
Honkala, I., Hudry, O., Lobstein, A.: On the ensemble of optimal dominating and
  locating-dominating codes in a graph.
\newblock Information Processing Letters \textbf{115}, 699--702 (2015)

\bibitem{hhltwinfree}
Honkala, I., Hudry, O., Lobstein, A.: On the ensemble of optimal identifying
  codes in a twin-free graph.
\newblock Cryptography and Communications \textbf{8}, 139--153 (2016)

\bibitem{horaobata}
Hora, A., Obata, N.: Quantum {P}robability and {S}pectral {A}nalysis of
  {G}raphs.
\newblock Springer, Berlin (2007)

\bibitem{krebsshaheen}
Krebs, M., Shaheen, A.: On the spectra of {J}ohnson graphs.
\newblock Electronic Journal of Linear Algebra \textbf{17}, 154--167 (2008)

\bibitem{distancedom}
Kreutzer, S., Ordyniak, S.: Distance {$d$}-domination games.
\newblock In: Graph-Theoretic Concepts in Computer Science, vol. 5911, pp.
  308--319. Springer, Berlin (2010)

\bibitem{liu}
Liu, C.: Introduction to {C}ombinatorial {M}athematics.
\newblock McGraw-Hill Book Co., New York (1968)

\bibitem{malikalidelta}
Malik, M.A., Ali, A.: The graph {$\Delta_{2n-1}$} is an induced subgraph of a
  {J}ohnson graph.
\newblock International Journal of Contemporary Mathematical Sciences
  \textbf{7}, 369--376 (2012)

\bibitem{malikali}
Malik, M.A., Ali, A.: Some results on induced subgraphs of {J}ohnson graphs.
\newblock International Mathematical Forum: Journal for Theory and Applications
  \textbf{7}, 445--454 (2012)

\bibitem{teshima}
Mynhardt, C., Teshima, L.: A note on some variations of the $\gamma$-graph.
\newblock Available at arXiv:1707.02039 [math.CO]

\bibitem{naimishaw}
Naimi, R., Shaw, J.: Induced subgraphs of {J}ohnson graphs.
\newblock Involve: A Journal of Mathematics \textbf{5}, 25--37 (2012)

\bibitem{ore}
Ore, O.: Theory of Graphs.
\newblock American Mathematical Society, Providence (1962)

\bibitem{5vgraphs}
de~Ridder, H., et~al.: Information system on graph classes and their
  inclusions.
\newblock Available at
  \url{http://www.graphclasses.org/smallgraphs.html#nodes5}

\bibitem{oeisconnected}
Sloane, N.: Number of connected graphs with {$n$} nodes, sequence {A}001349 in
  \emph{{T}he {O}n-{L}ine {E}ncyclopedia of {I}nteger {S}equences}.
\newblock Available at \url{http://oeis.org/A001349}

\bibitem{sarinduce}
Sridharan, N., Amutha, S., Rao, S.: Induced subgraphs of gamma graphs.
\newblock Discrete Mathematics, Algorithms and Applications \textbf{5} (2013)

\bibitem{sstrees}
Sridharan, N., Subramanian, K.: Trees and unicyclic graphs are
  {$\gamma$}-graphs.
\newblock Journal of Combinatorial Mathematics and Combinatorial Computing
  \textbf{69}, 231--236 (2009)

\bibitem{ssgamma}
Subramanian, K., Sridharan, N.: {$\gamma$}-graph of a graph.
\newblock Bulletin of Kerala Mathematics Association \textbf{5}, 17--34 (2008)

\end{thebibliography}

%\newpage
\section*{Appendix A: Classification of labellable graphs on at most five vertices}

In this appendix we classify the 31 connected graphs on at most five vertices~\cite{5vgraphs} as comprising 27 which are labellable, and 4 which are minimally unlabellable by the results of Section~\ref{sec:minunlab5}.
The letter labelling of the four minimally unlabellable graphs follows that shown in Theorem~\ref{thm:minunlab5}. Graphs with the same number of vertices are arranged (from top to bottom within each column) in increasing order of the number of edges.

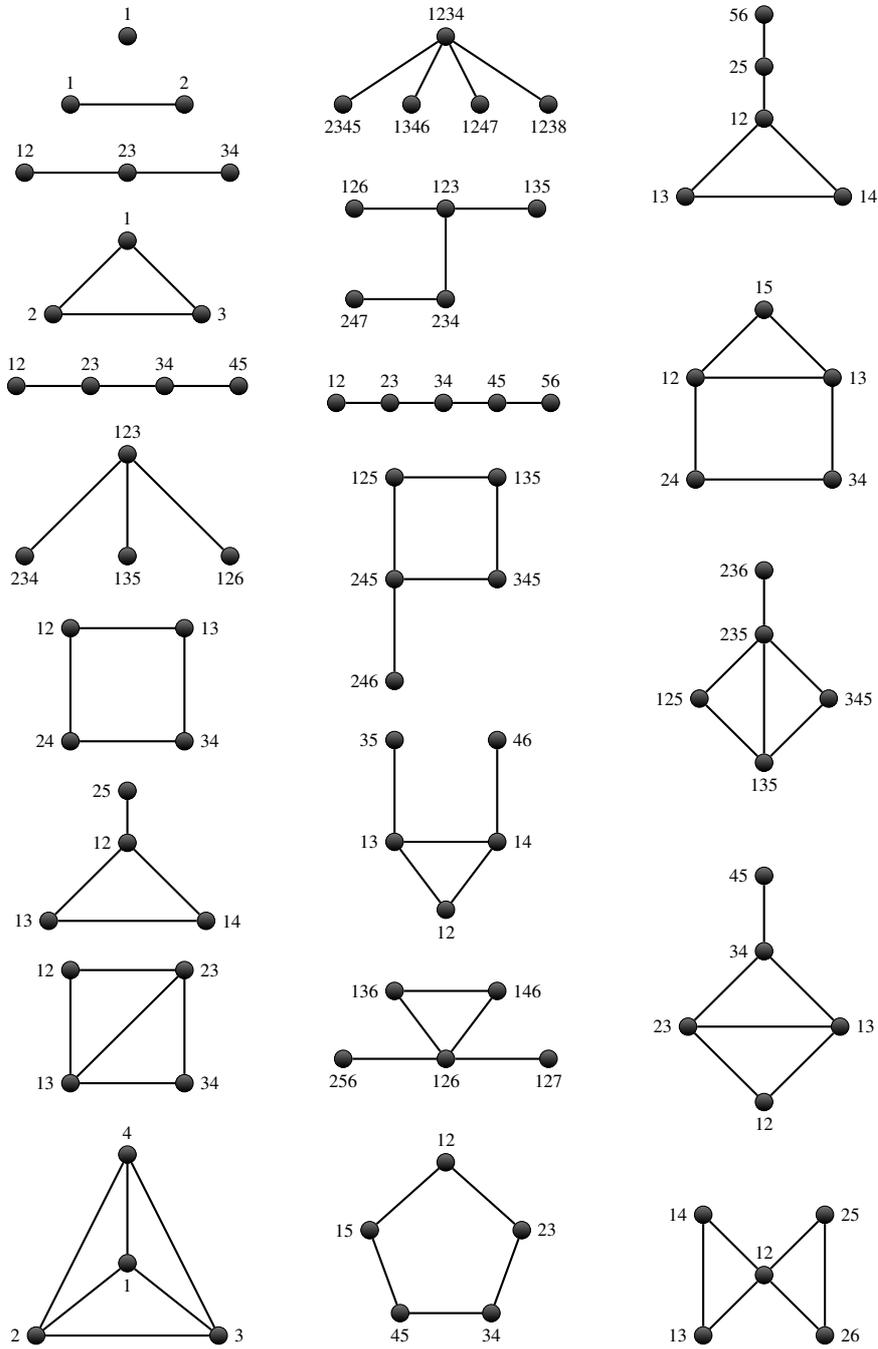
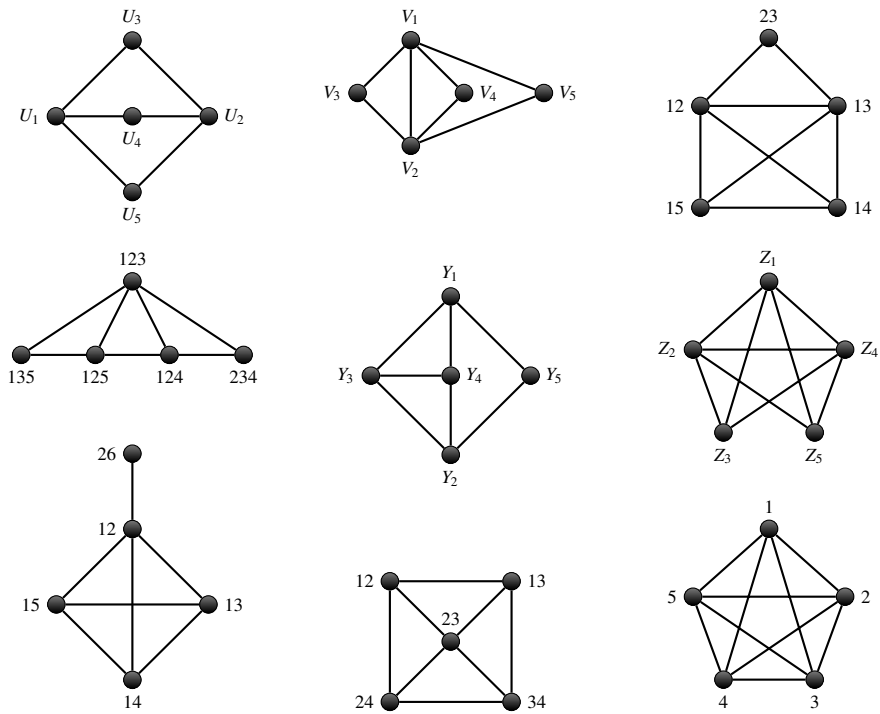
\begin{scriptsize}
\begin{multicols}{3}

%\subsection*{1 vertex} 

%1
\begin{center}
\begin{tikzpicture}[scale=1.0]
\node [n, label=above: {$1$}] (v1) at (2,2) {};
\end{tikzpicture}
\end{center}

%\subsection*{2 vertices} 

%2
\begin{center}
\begin{tikzpicture}[scale=1.0]
\node [n, label=above: {$1$}] (v1) at (0,0) {};
\node [n, label=above: {$2$}] (v2) at (1.5,0) {};
\draw [e] (v1) -- (v2);
\end{tikzpicture}
\end{center}

%\subsection*{3 vertices} 

%3
\begin{center}
\begin{tikzpicture}[scale=0.9]
\node [n, label=above: {$12$}] (v1) at (0,2) {};
\node [n, label=above: {$23$}] (v2) at (1.5,2) {};
\node [n, label=above: {$34$}] (v3) at (3,2) {};
\draw [e] (v1) -- (v2) -- (v3);
\end{tikzpicture}
\end{center}

%4
\begin{center}
\begin{tikzpicture}[scale=0.65]
\node [n, label=above: {$1$}] (v1) at (0,0) {};
\node [n, label=left: {$2$}] (v2) at (-1.5,-1.5) {};
\node [n, label=right: {$3$}] (v3) at (1.5,-1.5) {};
\draw [e] (v1) -- (v2) -- (v3) -- (v1);
\end{tikzpicture}
\end{center}

%\subsection*{4 vertices} 

%5
\begin{center}
\begin{tikzpicture}[scale=0.65]
\node [n, label=above: {$12$}] (v1) at (0,0) {};
\node [n, label=above: {$23$}] (v2) at (1.5,0) {};
\node [n, label=above: {$34$}] (v3) at (3,0) {};
\node [n, label=above: {$45$}] (v4) at (4.5,0) {};
\draw [e] (v1) -- (v2) -- (v3) -- (v4);
\end{tikzpicture}
\end{center}

%6
\begin{center}
\begin{tikzpicture}[scale=0.9]
\node [n, label=above: {$123$}] (v1) at (0,0) {};
\node [n, label=below: {$234$}] (v2) at (-1.5,-1.5) {};
\node [n, label=below: {$135$}] (v3) at (0,-1.5) {};
\node [n, label=below: {$126$}] (v4) at (1.5,-1.5) {};
\draw [e] (v2) -- (v1) -- (v3);
\draw [e] (v1) -- (v4);
\end{tikzpicture}
\end{center}

%7
\begin{center}
\begin{tikzpicture}[scale=1.0]
\node [n, label=left: {$12$}] (v1) at (0,1.5) {};
\node [n, label=right: {$13$}] (v2) at (1.5,1.5) {};
\node [n, label=right: {$34$}] (v3) at (1.5,0) {};
\node [n, label=left: {$24$}] (v4) at (0,0) {};
\draw [e] (v1) -- (v2) -- (v3) -- (v4) -- (v1);
\end{tikzpicture}
\end{center}

%8
\begin{center}
\begin{tikzpicture}[scale=0.69]
\node [n, label=left: {$12$}] (v1) at (0,0) {};
\node [n, label=left: {$13$}] (v2) at (-1.5,-1.5) {};
\node [n, label=right: {$14$}] (v3) at (1.5,-1.5) {};
\node [n, label=left: {$25$}] (v4) at (0,1) {};
\draw [e] (v1) -- (v2) -- (v3) -- (v1) -- (v4);
\end{tikzpicture}
\end{center}

%9
\begin{center}
\begin{tikzpicture}[scale=1.0]
\node [n, label=left: {$12$}] (v1) at (0,1.5) {};
\node [n, label=right: {$23$}] (v2) at (1.5,1.5) {};
\node [n, label=right: {$34$}] (v3) at (1.5,0) {};
\node [n, label=left: {$13$}] (v4) at (0,0) {};
\draw [e] (v1) -- (v2) -- (v3) -- (v4) -- (v1);
\draw [e] (v2) -- (v4);
\end{tikzpicture}
\end{center}

%10
\begin{center}
\begin{tikzpicture}[scale=0.8]
\node [n, label=below: {$1$}] (v1) at (0,-0.3) {};
\node [n, label=left: {$2$}] (v2) at (-1.5,-1.5) {};
\node [n, label=right: {$3$}] (v3) at (1.5,-1.5) {};
\node [n, label=above: {$4$}] (v4) at (0,1.5) {};
\draw [e] (v1) -- (v2) -- (v3) -- (v1) -- (v4) -- (v2);
\draw [e] (v4) -- (v3);
\end{tikzpicture}
\end{center}

%\subsection*{5 vertices} 

%11
\begin{center}
\begin{tikzpicture}[scale=0.6]
\node [n, label=above: {$1234$}] (v1) at (0,1.5) {};
\node [n, label=below: {$2345$}] (v2) at (-2.25,0) {};
\node [n, label=below: {$1346$}] (v3) at (-0.75,0) {};-
\node [n, label=below: {$1247$}] (v4) at (0.75,0) {};
\node [n, label=below: {$1238$}] (v5) at (2.25,0) {};
\draw [e] (v2) -- (v1) -- (v3);
\draw [e] (v4) -- (v1) -- (v5);
\end{tikzpicture}
\end{center}

%12
\begin{center}
\begin{tikzpicture}[scale=0.8]
\node [n, label=above: {$123$}] (v1) at (0,1.5) {};
\node [n, label=above: {$135$}] (v2) at (1.5,1.5) {};
\node [n, label=above: {$126$}] (v3) at (-1.5,1.5) {};
\node [n, label=below: {$234$}] (v4) at (0,0) {};
\node [n, label=below: {$247$}] (v5) at (-1.5,0) {};
\draw [e] (v2) -- (v1) -- (v3);
\draw [e] (v1) -- (v4) -- (v5);
\end{tikzpicture}
\end{center}

%13
\begin{center}
\begin{tikzpicture}[scale=0.47]
\node [n, label=above: {$12$}] (v1) at (0,0) {};
\node [n, label=above: {$23$}] (v2) at (1.5,0) {};
\node [n, label=above: {$34$}] (v3) at (3,0) {};
\node [n, label=above: {$45$}] (v4) at (4.5,0) {};
\node [n, label=above: {$56$}] (v5) at (6,0) {};
\draw [e] (v1) -- (v2) -- (v3) -- (v4) -- (v5);
\end{tikzpicture}
\vspace{1em}
\end{center}

%14
\begin{center}
\begin{tikzpicture}[scale=0.9]
\node [n, label=left: {$125$}] (v1) at (0,1.5) {};
\node [n, label=right: {$135$}] (v2) at (1.5,1.5) {};
\node [n, label=right: {$345$}] (v3) at (1.5,0) {};
\node [n, label=left: {$245$}] (v4) at (0,0) {};
\node [n, label=left: {$246$}] (v5) at (0,-1.5) {};
\draw [e] (v1) -- (v2) -- (v3) -- (v4) -- (v1);
\draw [e] (v4) -- (v5);
\end{tikzpicture}
\end{center}

%15
\begin{center}
\begin{tikzpicture}[scale=0.9]
\node [n, label=left: {$13$}] (v1) at (0,0) {};
\node [n, label=right: {$14$}] (v2) at (1.5,0) {};
\node [n, label=below: {$12$}] (v3) at (0.75,-1) {};
\node [n, label=left: {$35$}] (v4) at (0,1.5) {};
\node [n, label=right: {$46 $}] (v5) at (1.5,1.5) {};
\draw [e] (v1) -- (v2) -- (v3) -- (v1) -- (v4);
\draw [e] (v2) -- (v5);
\end{tikzpicture}
\end{center}

%16
\begin{center}
\begin{tikzpicture}[scale=0.9]
\node [n, label=left: {$136$}] (v1) at (0,0) {};
\node [n, label=right: {$146$}] (v2) at (1.5,0) {};
\node [n, label=below: {$126$}] (v3) at (0.75,-1) {};
\node [n, label=below: {$256$}] (v4) at (-0.75,-1) {};
\node [n, label=below: {$127$}] (v5) at (2.25,-1) {};
\draw [e] (v1) -- (v2) -- (v3) -- (v1);
\draw [e] (v4) -- (v3) -- (v5);
\end{tikzpicture}
\end{center}

%17
\begin{center}
\begin{tikzpicture}[scale=1.0]
\node [n, label=left: {$15$}] (v1) at (-1,0.1) {};
\node [n, label=right: {$23$}] (v2) at (1,0.1) {};
\node [n, label=above: {$12$}] (v3) at (0,1) {};
\node [n, label=below: {$45$}] (v4) at (-0.6,-1) {};
\node [n, label=below: {$34$}] (v5) at (0.6,-1) {};
\draw [e] (v2) -- (v3) -- (v1) --(v4) -- (v5) -- (v2);
\end{tikzpicture}
\end{center}

%18
\begin{center}
\begin{tikzpicture}[scale=0.69]
\node [n, label=left: {$13$}] (v1) at (-0.5,0) {};
\node [n, label=right: {$14$}] (v2) at (2.5,0) {};
\node [n, label=left: {$12$}] (v3) at (1,1.5) {};
\node [n, label=left: {$25$}] (v4) at (1,2.5) {};
\node [n, label=left: {$56$}] (v5) at (1,3.5) {};
\draw [e] (v3) -- (v2) -- (v1) -- (v3) -- (v4) -- (v5);
\end{tikzpicture}
\end{center}

%19
\begin{center}
\begin{tikzpicture}[scale=0.9]
\node [n, label=left: {$12$}] (v1) at (0,0) {};
\node [n, label=right: {$13$}] (v2) at (2,0) {};
\node [n, label=above: {$15$}] (v3) at (1,1) {};
\node [n, label=left: {$24$}] (v4) at (0,-1.5) {};
\node [n, label=right: {$34$}] (v5) at (2,-1.5) {};
\draw [e] (v1) -- (v2) -- (v3) -- (v1) --(v4) -- (v5) -- (v2);
\end{tikzpicture}
\end{center}

%20
\begin{center}
\begin{tikzpicture}[scale=0.85]
\node [n, label=left: {$125$}] (v1) at (-1,1) {};
\node [n, label=left: {$235$}] (v2) at (0,2) {};
\node [n, label=right: {$345$}] (v3) at (1,1) {};
\node [n, label=below: {$135$}] (v4) at (0,0) {};
\node [n, label=left: {$236$}] (v5) at (0,3) {};
\draw [e] (v1) -- (v2) -- (v3) -- (v4) -- (v1);
\draw [e] (v5) -- (v2) -- (v4);
\end{tikzpicture}
\end{center}

%21
\begin{center}
\begin{tikzpicture}[scale=1.0]
\node [n, label=left: {$34$}] (v1) at (0,2) {};
\node [n, label=right: {$13$}] (v2) at (1,1) {};
\node [n, label=below: {$12$}] (v3) at (0,0) {};
\node [n, label=left: {$23$}] (v4) at (-1,1) {};
\node [n, label=left: {$45$}] (v5) at (0,3) {};
\draw [e] (v5) -- (v1) -- (v2) -- (v3) -- (v4) -- (v1);
\draw [e] (v2) -- (v4);
\end{tikzpicture}
\end{center}

%22
\begin{center}
\begin{tikzpicture}[scale=0.8]
\node [n, label=left: {$14$}] (v1) at (0,2) {};
\node [n, label=right: {$25$}] (v2) at (2,2) {};
\node [n, label=right: {$26$}] (v3) at (2,0) {};
\node [n, label=left: {$13$}] (v4) at (0,0) {};
\node [n, label=above: {$12$}] (v5) at (1,1) {};
\draw [e] (v4) -- (v1) -- (v5) -- (v3) -- (v2) -- (v5) -- (v4);
\end{tikzpicture}
\end{center}

%23
\begin{center}
\begin{tikzpicture}[scale=0.67]
\node [n, label=below: {$U_5$}] (v1) at (0,0) {};
\node [n, label=left: {$U_1$}] (v2) at (-1.5,1.5) {};
\node [n, label=below: {$U_4$}] (v3) at (0,1.5) {};
\node [n, label=right: {$U_2$}] (v4) at (1.5,1.5) {};
\node [n, label=above: {$U_3$}] (v5) at (0,3) {};
\draw [e] (v1) -- (v2) -- (v5) -- (v4) -- (v1);
\draw [e] (v2) -- (v3) -- (v4);
\end{tikzpicture} 
\end{center}

%24
\begin{center}
\begin{tikzpicture}[scale=0.65]
\node [n, label=above: {$123$}] (v1) at (0,1.5) {};
\node [n, label=below: {$135$}] (v2) at (-2.25,0) {};
\node [n, label=below: {$125$}] (v3) at (-0.75,0) {};
\node [n, label=below: {$124$}] (v4) at (0.75,0) {};
\node [n, label=below: {$234$}] (v5) at (2.25,0) {};
\draw [e] (v2) -- (v1) -- (v3);
\draw [e] (v4) -- (v1) -- (v5) -- (v4) -- (v3) -- (v2);
\end{tikzpicture}
\end{center}
\vspace{1em}

%25
\begin{center}
\begin{tikzpicture}[scale=1.0]
\node [n, label=left: {$12$}] (v1) at (0,2) {};
\node [n, label=right: {$13$}] (v2) at (1,1) {};
\node [n, label=below: {$14$}] (v3) at (0,0) {};
\node [n, label=left: {$15$}] (v4) at (-1,1) {};
\node [n, label=left: {$26$}] (v5) at (0,3) {};
\draw [e] (v5) -- (v1) -- (v2) -- (v3) -- (v4) -- (v1) -- (v3);
\draw [e] (v2) -- (v4);
\end{tikzpicture}
\end{center}

%26
\begin{center}
\begin{tikzpicture}[scale=0.7]
\node [n, label=below: {$V_2$}] (v1) at (0,0) {};
\node [n, label=left: {$V_3$}] (v2) at (-1,1) {};
\node [n, label=right: {$V_4$}] (v3) at (1,1) {};
\node [n, label=above: {$V_1$}] (v4) at (0,2) {};
\node [n, label=right: {$V_5$}] (v5) at (2.5,1) {};
\draw [e] (v1) -- (v2) -- (v4) -- (v3) -- (v1) -- (v4) -- (v5) -- (v1);
\end{tikzpicture}
\end{center}

%27
\begin{center}
\begin{tikzpicture}[scale=0.7]
\node [n, label=below: {$Y_2$}] (v1) at (0,0) {};
\node [n, label=left: {$Y_3$}] (v2) at (-1.5,1.5) {};
\node [n, label=right: {$Y_4$}] (v3) at (0,1.5) {};
\node [n, label=right: {$Y_5$}] (v4) at (1.5,1.5) {};
\node [n, label=above: {$Y_1$}] (v5) at (0,3) {};
\draw [e] (v1) -- (v3) -- (v5) -- (v4) -- (v1) -- (v2) -- (v5);
\draw [e] (v2) -- (v3);
\end{tikzpicture} 
\end{center}

%28
\begin{center}
\begin{tikzpicture}[scale=0.8]
\node [n, label=left: {$12$}] (v1) at (0,2) {};
\node [n, label=right: {$13$}] (v2) at (2,2) {};
\node [n, label=right: {$34$}] (v3) at (2,0) {};
\node [n, label=left: {$24$}] (v4) at (0,0) {};
\node [n, label=above: {$23$}] (v5) at (1,1) {};
\draw [e] (v2) -- (v5) -- (v4);
\draw [e] (v3) -- (v5) -- (v1) -- (v2) -- (v3) -- (v4) -- (v1);
\end{tikzpicture}
\end{center}

%29
\begin{center}
\begin{tikzpicture}[scale=0.9]
\node [n, label=left: {$12$}] (v1) at (-0.5,0) {};
\node [n, label=right: {$13$}] (v2) at (1.5,0) {};
\node [n, label=above: {$23$}] (v3) at (0.5,1) {};
\node [n, label=left: {$15$}] (v4) at (-0.5,-1.5) {};
\node [n, label=right: {$14$}] (v5) at (1.5,-1.5) {};
\draw [e] (v4) -- (v1) -- (v3) -- (v2) -- (v4) -- (v5) -- (v1) -- (v2) -- (v5);
\end{tikzpicture}
\end{center}

%30
\begin{center}
\begin{tikzpicture}[scale=1.0]
\node [n, label=left: {$Z_2$}] (v1) at (-1,0.1) {};
\node [n, label=right: {$Z_4$}] (v2) at (1,0.1) {};
\node [n, label=above: {$Z_1$}] (v3) at (0,1) {};
\node [n, label=below: {$Z_3$}] (v4) at (-0.6,-1) {};
\node [n, label=below: {$Z_5$}] (v5) at (0.6,-1) {};
\draw [e] (v5) -- (v2) -- (v4) -- (v3) -- (v5) -- (v1) -- (v2) -- (v3) -- (v1) --(v4);
\end{tikzpicture}
\end{center}

%31
\begin{center}
\begin{tikzpicture}[scale=1.0]
\node [n, label=left: {$5$}] (v1) at (-1,0.1) {};
\node [n, label=right: {$2$}] (v2) at (1,0.1) {};
\node [n, label=above: {$1$}] (v3) at (0,1) {};
\node [n, label=below: {$4$}] (v4) at (-0.6,-1) {};
\node [n, label=below: {$3$}] (v5) at (0.6,-1) {};
\draw [e] (v2) -- (v3) -- (v1) --(v4) -- (v5) -- (v2) -- (v4) -- (v3) -- (v5) -- (v1) -- (v2);
\end{tikzpicture}
\end{center}

\end{multicols}
\end{scriptsize}

%\newpage
\section*{Appendix B: Classification of labellable graphs on six vertices}

In this appendix we classify the 112 connected graphs on six vertices~\cite{cvetpet} as comprising:
69 which are labellable, as demonstrated;
39 which are unlabellable because they contain as a proper induced subgraph one of the four minimally unlabellable graphs on five vertices (indicated using $U_i, V_i, Y_i$, or $Z_i$ as shown in Theorem~\ref{thm:minunlab5});
and 4 which are minimally unlabellable by the results of Section~\ref{sec:minunlab6} (left unlabelled).
The graphs are arranged (from top to bottom within each column) in increasing order of the number of edges.

\begin{scriptsize}
\begin{multicols}{3}

%1
\begin{center}

\end{center}

\end{multicols}
\end{scriptsize}

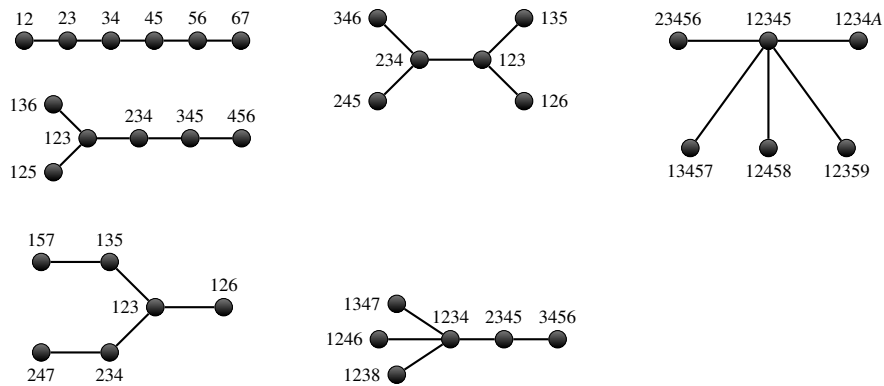
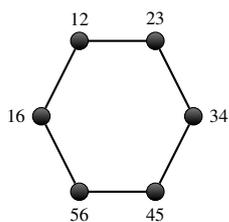
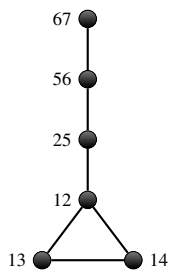
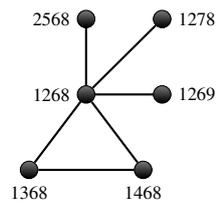
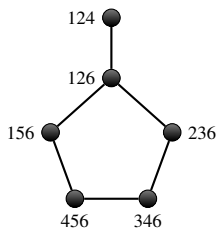
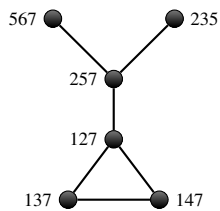
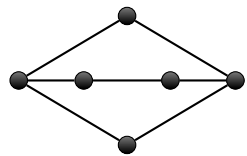
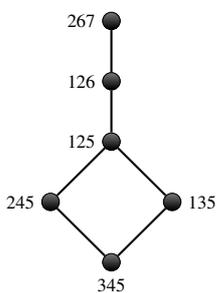
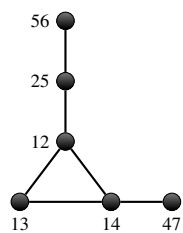
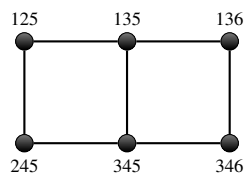
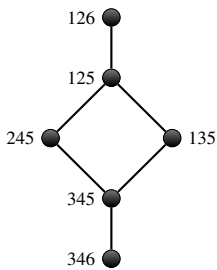
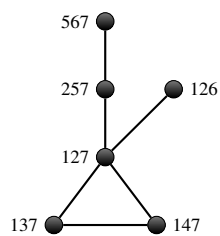
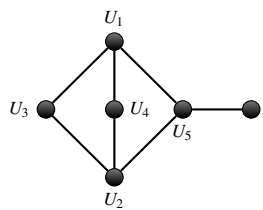
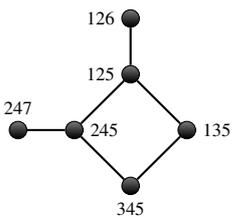
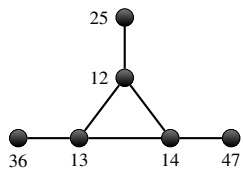
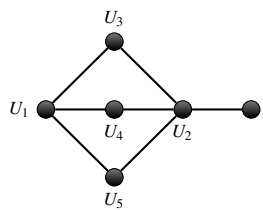
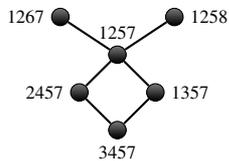
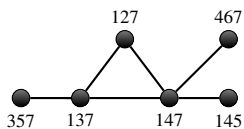
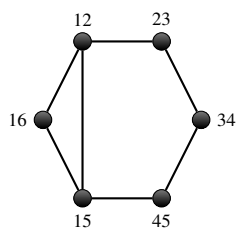
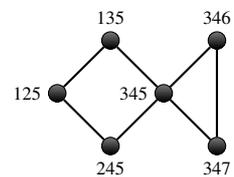
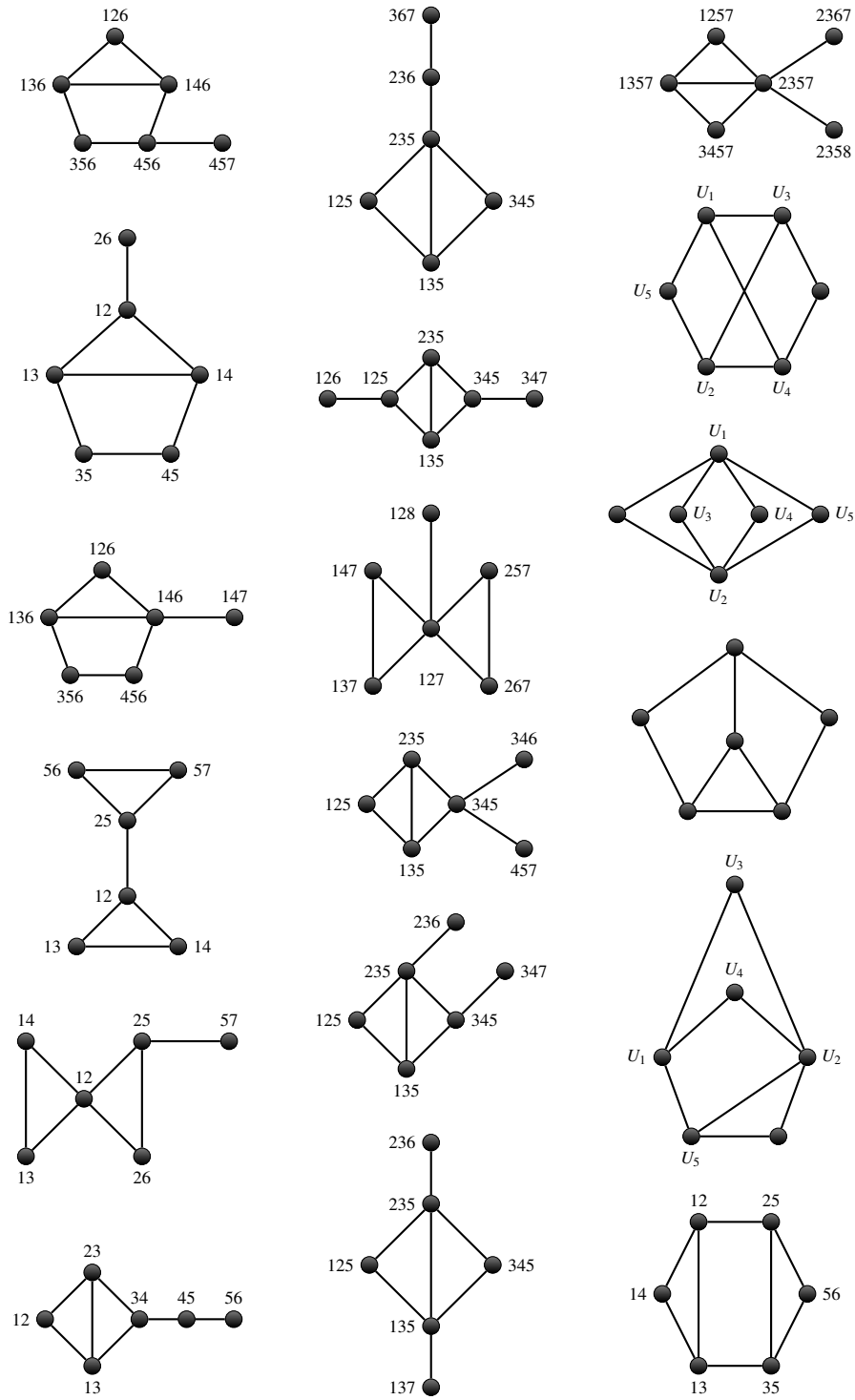
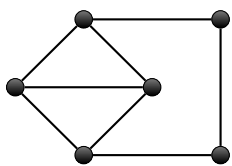
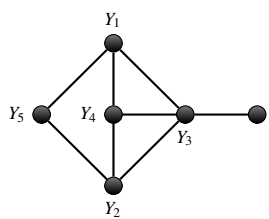
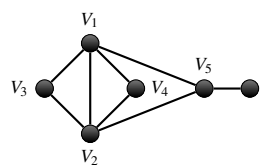
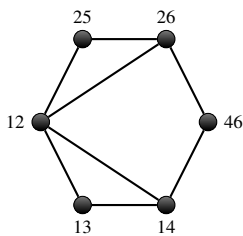
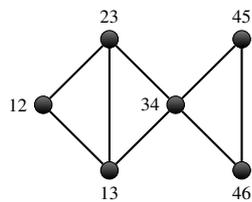
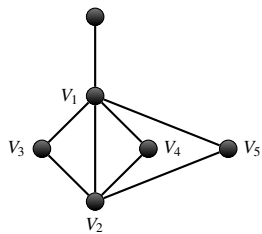
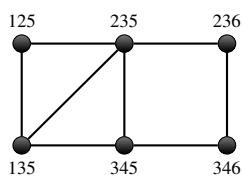
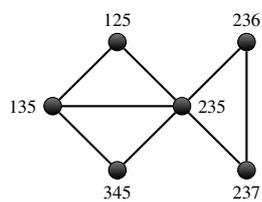
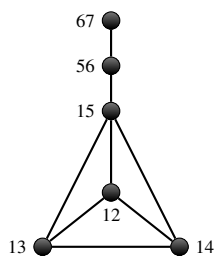
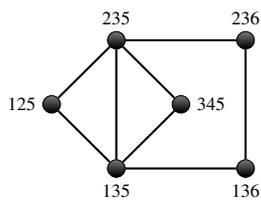
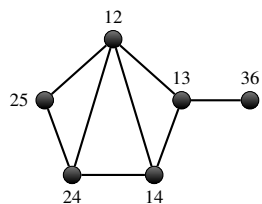
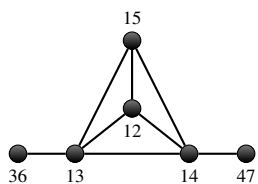
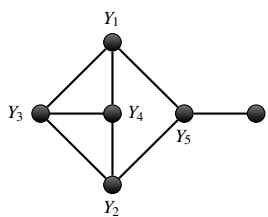
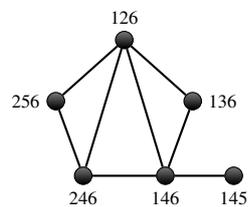
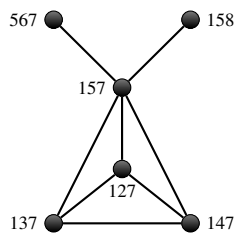
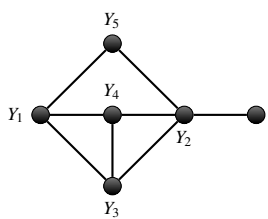
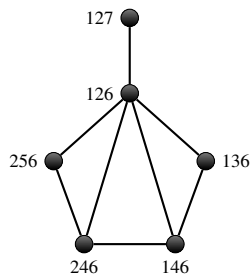
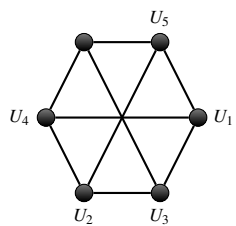
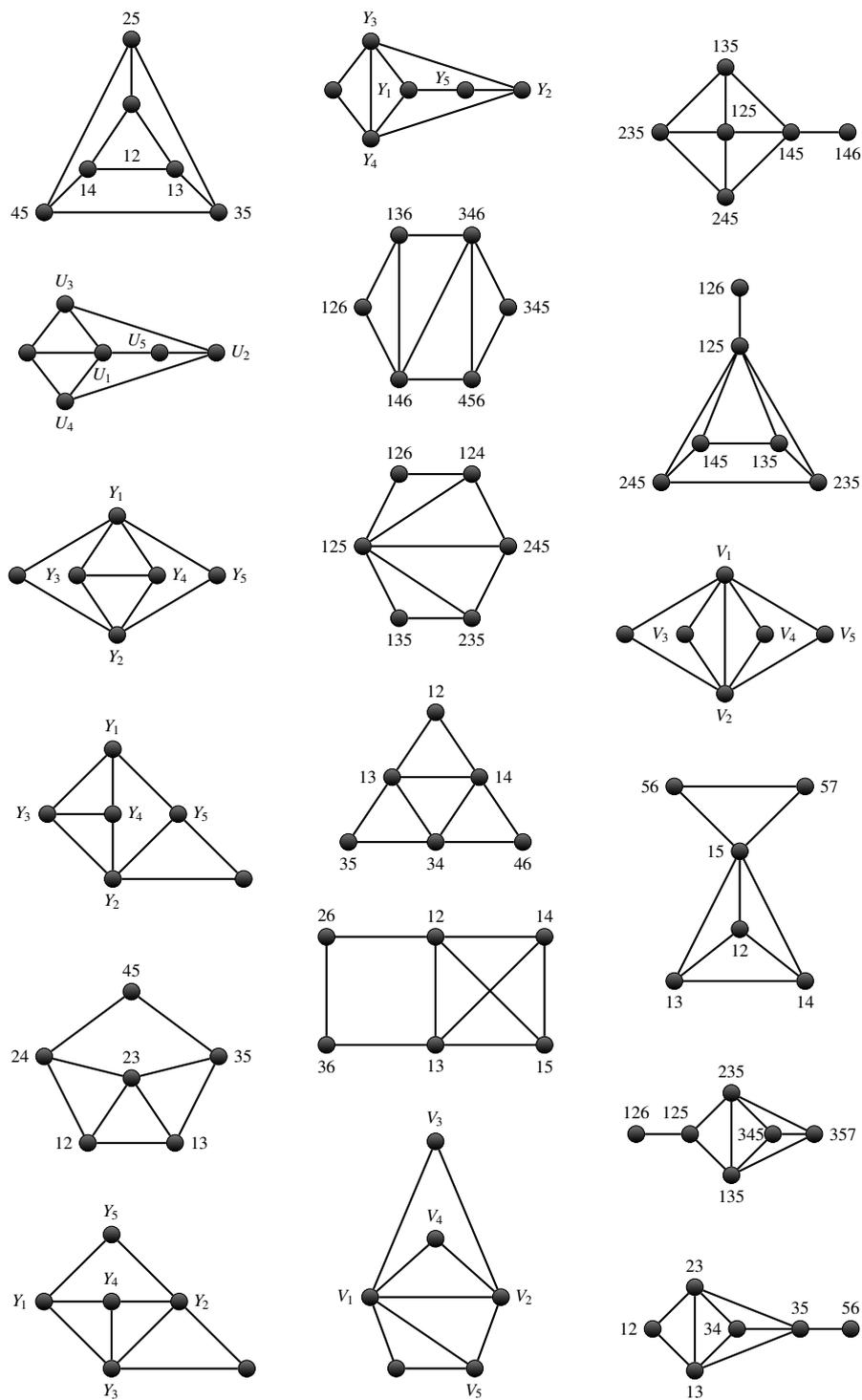
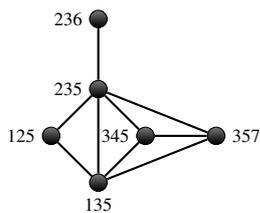
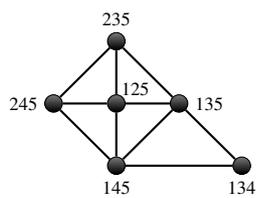
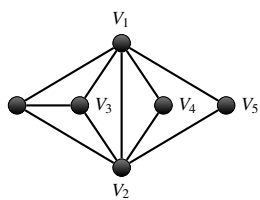
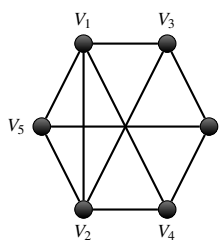
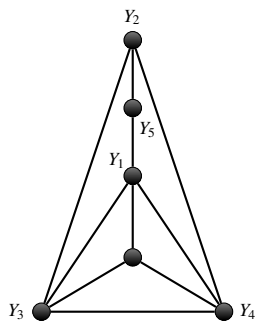
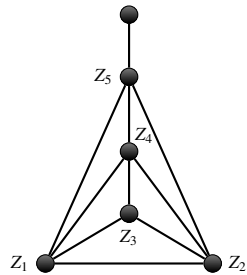
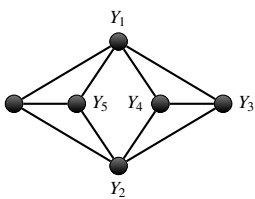
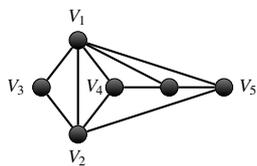
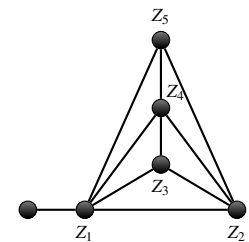
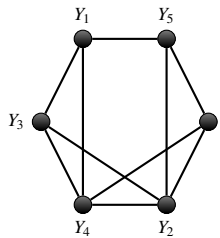
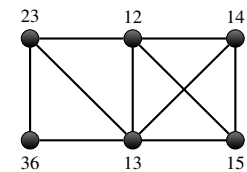
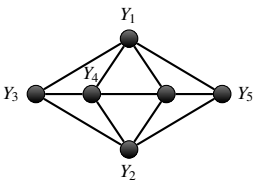
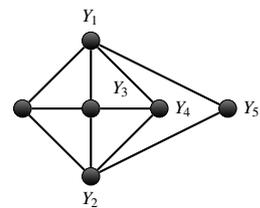
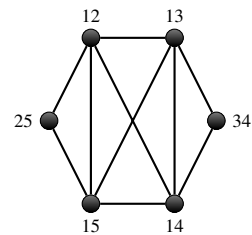
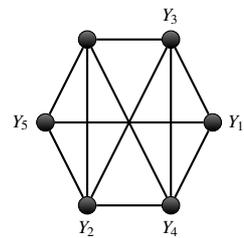
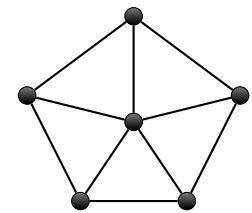
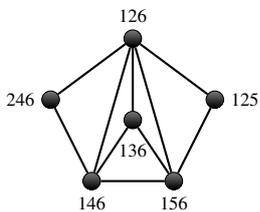
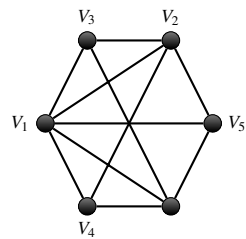
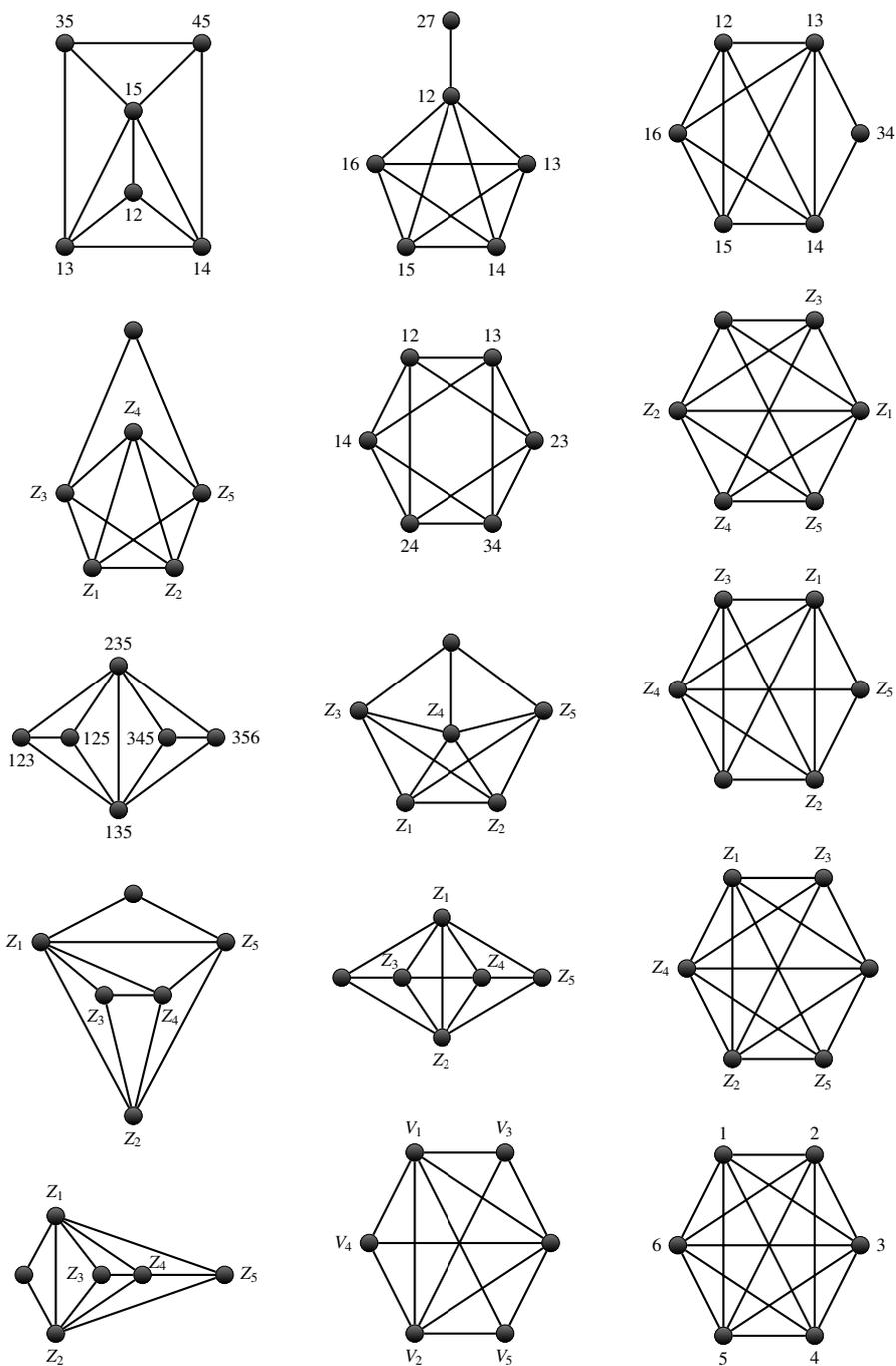
\end{document}